\documentclass{amsart}

\usepackage{graphicx}%
\usepackage{amsmath,amssymb,amsfonts}%
\usepackage{xcolor}%
\usepackage{enumerate}%
\usepackage{prettyref}%

\newrefformat{cor}{Corollary~\ref{#1}}
\newrefformat{subsec}{Section~\ref{#1}}
\newrefformat{lem}{Lemma~\ref{#1}}
\newrefformat{thm}{Theorem~\ref{#1}}
\newrefformat{sec}{Section~\ref{#1}}
\newrefformat{chap}{Chapter~\ref{#1}}
\newrefformat{prop}{Proposition~\ref{#1}}
\newrefformat{exa}{Example~\ref{#1}}
\newrefformat{tab}{Table~\ref{#1}}
\newrefformat{rem}{Remark~\ref{#1}}
\newrefformat{def}{Definition~\ref{#1}}
\newrefformat{fig}{Figure~\ref{#1}}
\newrefformat{claim}{Claim~\ref{#1}}
\newrefformat{assu}{Assumption~\ref{#1}}

\numberwithin{equation}{section}

\theoremstyle{plain}
\newtheorem{thm}{Theorem}[section]
\newtheorem{cor}[thm]{Corollary}
\newtheorem{lem}[thm]{Lemma}
\theoremstyle{definition}

\newtheorem{rem}[thm]{Remark}

\begin{document}

\title[stability of network-outputs]{The role of parameter Jacobians in the stability of network outputs}

\author[H. Jeong et al.]{Halyun Jeong}
\address{Department of Mathematics \& Statistics, University at Albany, State University of New York, Albany, NY, USA}

\author[]{Palle E.\,T. Jorgensen}
\address{Department of Mathematics, The University of Iowa, Iowa City, IA 52242, USA}

\author[]{Hyun-Kyoung Kwon}
\address{Department of Mathematics \& Statistics, University at Albany, State University of New York, Albany, NY, USA}

\author[]{Myung-Sin Song}
\address{Department of Mathematics \& Statistics, Southern Illinois University Edwardsville, Edwardsville, IL, USA}

\author[]{James Tian}
\address{Mathematical Reviews, 535 W William St, Ste 210, Ann Arbor, MI, USA}

\begin{abstract}
In the framework of network dynamics, learning models, and neural tangent kernels (NTK), we show that the corresponding linearized dynamics leads naturally to a semigroup formulation. More precisely, in our analysis of input/output models, the time-dynamics is presented via special semigroups of linear operators on Hilbert spaces, together with an associated class of semigroup perturbations. In this context, we then present new and explicit a priori perturbation-bound results: for the fixed-kernel linearization constructions arising in the NTK setting, we prove norm-bounds on the corresponding semigroup perturbations, in the form of explicit finite-time perturbation estimates. We further present refinements on prescribed task spaces, Ces\`aro-averaged (ergodic) comparisons estimates, and versions in which the lower spectral edge assumption is replaced by a spectral-distribution condition. We also extend the comparison to nonautonomous NTK evolutions through piecewise-frozen approximations, record a corresponding discrete Euler specialization, and offer worked examples in order to illustrate our perturbation-bound estimates.
\end{abstract}



\maketitle

\section{Introduction}\label{sec:1}

The role of Jacobians in stability questions for network outputs is
the main theme of this paper. We study bounded linearizations between
Hilbert spaces and use operator-theoretic tools to quantify how perturbations
on the parameter side propagate to the output side. Our focus is on
sensitivity, stability, and approximation error for network outputs
when one restricts, freezes, or prunes parameter directions. For the
general context, motivation, and background we refer to
\cite{MR3711461,MR3144003,MR2968753,MR5040882,MR5038664,MR5009653,MR5008590}.

The present setup may be viewed as the fixed-kernel linearization
that appears in the neural tangent kernel picture. In that interpretation,
$H$ is the parameter space, $K$ is the output space, and $T:H\to K$
is the Jacobian of the output map at a reference point. The associated
output-side operator is $G=TT^{*}$, and for squared loss, the residual
in the linearized model evolves under the contraction semigroup
$e^{-tG}$. Replacing $T$ by $TP$ means that only the parameter
directions in $ran P$ remain active in the linearized
dynamics. The corresponding effective operator is
$G_{P}=TPT^{*}$, with semigroup
$S_{P}\left(t\right)=e^{-tG_{P}}$. The basic problem studied here
is to compare this perturbed evolution with the original one on a
prescribed task space and over a prescribed time interval.

We fix a closed subspace $M\subseteq K$ to isolate
the output directions relevant to the problem at hand. Thus the
comparison is not made on all of $K$, but only on the task space $M$
and the closed $G$-cyclic subspace it generates. This
restriction is important: pruning is not only a model-reduction step
on the parameter side, but also a perturbation of the output-side
generator that drives the residual dynamics. The main idealization is
that the linearized operator is held fixed throughout the evolution.

Our choice of an approximation scheme and a priori estimates is also
motivated in part by the dynamical sampling viewpoint, where stability
and reconstruction are studied through samples of operator orbits
generated by an underlying evolution; see, for example,
\cite{MR5071919, MR4581908, MR4696783}. In the discrete-time setting, this typically involves
iterates such as matrix powers $A^{n}$, while in continuous time, the corresponding
evolution is described by a semigroup, 
$S(t)=e^{-tG}$ or $S_{P}(t)=e^{-tG_{P}}$.

Two quantitative tangent-loss hypotheses on the pruning projection $P$
are given in the paper. When
the removed parameter directions contain only a small relative portion
of the task-relevant tangent image, we use the relative task-capture
condition
\[
\|(I-P)T^*u\|_H \le \epsilon \|T^*u\|_H
\]
on the $G$-cyclic task space generated by $M$. This condition
is well suited to finite-time stability estimates, because it says that
on the spectral directions coupled to the task, the discarded
parameter directions carry only a small part of the tangent energy. It
gives the comparison estimates of \prettyref{sec:3}. The complementary
finite-dimensional comparison in \prettyref{sec:4} does not assume this
condition; it instead uses the invariance of the common task space under
$G_{P}$ and separates the diagonal attenuation from off-diagonal mixing.

For exact hard truncation, however, the discarded tangent component often need
not be uniformly small relative to \(\|T^*u\|_H\) on all directions in the
cyclic task space. In that case we use instead the mixed condition
\[
\|(I-P)T^*u\|_H^2 \le \alpha \|T^*u\|_H^2+\beta\|u\|_K^2,
\]
in which \(\alpha\) measures a relative loss while \(\beta\) records an
additive task-space error. 

This should be viewed as a relative form bound with an additive
tolerance; the term \(\beta\|u\|_K^2\) is useful when the task energy
\(\|T^*u\|_H^2\) does not uniformly control the ambient task norm.
Further motivation based on the machine learning and signal processing literature is given in
\prettyref{rem:mixed-task-tail-interpretation}. In the
Ces\`aro/ergodic part of the paper, the same condition is applied along
the unpruned orbit, where, together with spectral information on a
chosen $G$-invariant task space, it gives quantitative decay estimates for the
averaged discarded tangent component.

The finite-time and averaged estimates measure different effects. The
finite-time bounds control the deviation of the two evolutions at each
time in a fixed interval. The Ces\`aro/ergodic estimates compare their
time averages over increasing horizons: they quantify approximations of
the original average by the pruned average and isolate the stationary
task component that may remain in the limit. In this averaged part of
the paper, the mixed condition quantifies tangent leakage along the
unpruned orbit.

We now briefly place this fixed-kernel viewpoint in the broader NTK literature.
The Neural Tangent Kernel (NTK) framework has become a standard tool
in the theory of wide neural networks, because it describes
gradient-based training through the tangent features determined by the
network Jacobian \cite{jacot2018neural,lee2019wide}. At a reference
parameter, often the initialization, the Jacobian defines a first-order
linearization of the network, and the associated tangent kernel governs
the corresponding output-space gradient flow. In the infinite-width
limit, under the usual scaling assumptions, the empirical NTK converges
to a deterministic limiting kernel which remains effectively constant
during training. The nonlinear training dynamics are then asymptotically
described by the gradient flow of the linearized model.

This is the fixed-kernel, or lazy-training, viewpoint: the parameters
remain close to the reference point and the network is well approximated
by its first-order linearization \cite{chizat2019lazy}. For squared
loss, the residual flow is therefore generated by \(G=TT^{*}\), where
\(T\) is the Jacobian at the reference parameter. Our analysis takes this
frozen-kernel model as its starting point and asks how the output
dynamics change when the tangent feature map is modified. If the
trainable parameter directions are restricted by a projection \(P\),
then \(T\) is replaced by \(TP\), and the effective output-side kernel
becomes \(G_{P}=TPT^{*}\). The problem studied here is therefore the
comparison of \(e^{-tTT^{*}}\) and \(e^{-tTPT^{*}}\) on a prescribed
task space. This viewpoint is also consistent with the broader idea that
effective learning may take place in low-dimensional or restricted
parameter subspaces \cite{aghajanyan2021intrinsic}, but our focus is the
resulting operator-theoretic perturbation problem. For more background
and additional references on tangent kernels and related operator-theoretic
tools, we refer readers to
\cite{MR4749130, MR4852244,MR4777426,MR5052480,MR5035464,MR4992731,MR4957888,MR4867336}.

The main contribution is an a priori, task-local perturbation theory for
output dynamics generated by Jacobian Gram operators. The contributions
may be summarized as follows; First, for a frozen Jacobian, we derive
explicit finite-time comparison bounds under task-local tangent capture,
together with a complementary finite-dimensional comparison that separates
attenuation within spectral modes coupled to the task from mixing between them.
Second, for Ces\`aro averages, we prove that the restricted average
approximates the original one at an inverse-horizon rate under a positive
lower spectral edge assumption, and we
identify exactly the stationary task component that persists in general.
Third, we replace the lower-edge assumption by taskwise bounds on the
low-frequency spectral mass of both the original and restricted generators
and obtain corresponding approximation rates. Fourth, we show that the basic finite-time
argument extends with little change to a class of unbounded generators under
suitable domain assumptions; a systematic unbounded theory is left for
future work since it could deviate from our main goal.  Finally, for time-dependent Jacobians, we control
piecewise-frozen approximations with local restrictions of parameter
directions and derive a corresponding discrete Euler estimate. 

\subsection{Organization}

Section~\ref{sec:2} derives the fixed-kernel output dynamics.
Sections~\ref{sec:3} and \ref{sec:4} establish finite-time comparison
estimates, first from task-local tangent capture and then through a
complementary finite-dimensional refinement that separates attenuation
from mixing. Sections~\ref{sec:ergodic-task} and
\ref{sec:spectral-distribution-replacement} study Ces\`aro-averaged
dynamics. They give approximation rates under either a positive lower
spectral edge or weaker low-frequency spectral-distribution assumptions
and identify the stationary part that can remain in the limiting
difference. Section~\ref{sec:unbounded-kato} shows that the basic
finite-time argument extends, under suitable domain assumptions, to a
class of unbounded generators. Sections~\ref{sec:na-ntk} and
\ref{subsec:na-discrete} treat time-dependent kernels through
piecewise-frozen and explicit Euler approximations. The Appendix gives a
finite-dimensional worked example with explicit constants.

\section{The fixed-kernel linearization and the NTK picture}\label{sec:2}

We briefly explain the machine learning interpretation of the operator-theoretic
model used in this paper. The aim is not to survey the full neural
tangent kernel literature, but to point out the linearized dynamics
that lead naturally to the semigroups studied below.
Let $H$ be a Hilbert space of parameters, let $K$ be a Hilbert space
of outputs, and let
\[
F:H\to K
\]
be a differentiable map representing a model. Fix a reference parameter
$\theta_{0}\in H$, and set
\[
T:=DF\left(\theta_{0}\right):H\to K.
\]
The first-order approximation of the model at $\theta_{0}$ is
\[
F\left(\theta_{0}+h\right)\approx F\left(\theta_{0}\right)+Th,
\]
for small parameter increments $h\in H$.
We restrict attention to the linearized dynamics at the reference
point $\theta_{0}$, with the Jacobian $T=DF\left(\theta_{0}\right)$
held fixed throughout the evolution. In the machine learning literature,
this is often called the fixed-kernel or frozen NTK setting. The analysis
below concerns this linearized residual flow, rather than the full
nonlinear training process.
Now fix a target output $y\in K$, and consider the squared loss
\[
\mathcal{L}\left(h\right):=\frac{1}{2}\left\Vert F\left(\theta_{0}\right)+Th-y\right\Vert ^{2}.
\]
Its gradient with respect to the parameter variable $h$ is
\[
\nabla\mathcal{L}\left(h\right)=T^{*}\left(F\left(\theta_{0}\right)+Th-y\right).
\]
Hence the gradient flow equation for the linearized model is
\[
h'\left(t\right)=-T^{*}\left(F\left(\theta_{0}\right)+Th\left(t\right)-y\right).
\]
It is more convenient to pass to the residual
\[
r\left(t\right):=F\left(\theta_{0}\right)+Th\left(t\right)-y\in K.
\]
Differentiating and using the gradient flow equation gives
\[
r'\left(t\right)=Th'\left(t\right)=-TT^{*}r\left(t\right).
\]
Thus, the residual evolves according to
\[
r'\left(t\right)=-Gr\left(t\right),\quad G:=TT^{*}.
\]
Since $G$ is positive and self-adjoint, it follows that
\[
r\left(t\right)=e^{-tG}r\left(0\right),\qquad t\ge0.
\]
This is the origin of the semigroup
\[
S\left(t\right):=e^{-tG}
\]
used throughout the paper.
In the neural tangent kernel picture, the operator $G=TT^{*}$ is
the output side kernel operator induced by the Jacobian $T$. In finite-dimensional
settings, it is the Gram operator associated with the linearized feature
map. Thus the semigroup $e^{-tG}$ describes the decay of residuals
under gradient flow for the frozen linearized model.

We next explain the role of the projection $P$. Suppose that only
the parameter directions in a closed subspace of $H$ are allowed
to evolve, while the remaining directions are frozen. If $P$ denotes
the orthogonal projection onto the active parameter subspace, then
the restricted linearized model becomes
\[
F\left(\theta_{0}\right)+TPh,
\]
and the corresponding gradient flow is
\[
h_{P}'\left(t\right)=-PT^{*}\left(F\left(\theta_{0}\right)+TPh_{P}\left(t\right)-y\right).
\]
If we set
\[
r_{P}\left(t\right):=F\left(\theta_{0}\right)+TPh_{P}\left(t\right)-y,
\]
then
\[
r_{P}'\left(t\right)=TPh_{P}'\left(t\right)=-TPT^{*}r_{P}\left(t\right).
\]
Hence the restricted residual dynamics are generated by
\[
G_{P}:=TPT^{*},
\]
and
\[
r_{P}\left(t\right)=e^{-tG_{P}}r_{P}\left(0\right).
\]
This is the origin of the perturbed semigroup
\[
S_{P}\left(t\right):=e^{-tG_{P}}.
\]
Accordingly, the problem studied in this paper is the comparison of
the two residual evolutions
\[
e^{-tTT^{*}}\qquad\text{and}\qquad e^{-tTPT^{*}},
\]
that is, the comparison between full linearized training and training
in which only a prescribed family of parameter directions remains
active.
The analysis below does not address the full nonlinear training process,
in which the Jacobian itself changes with time. Instead, it isolates
the stability question inside the frozen linearization at a reference
point. This is nevertheless useful because it yields an explicit model
in which the effect of restricting parameter directions can be studied
in operator-theoretic terms.
We now explain the role of the task space $M$. In many situations,
one is not interested in controlling the full output space $K$, but
only a distinguished family of output directions relevant to the task
under consideration. These may represent, for example, a prescribed
collection of residuals, observables, labels, or output modes. For
this reason, the comparison between $S\left(t\right)$ and $S_{P}\left(t\right)$
is made only on a fixed closed subspace
\[
M\subseteq K.
\]
Since the unperturbed flow may move vectors in $M$ through additional
output directions, the natural space on which to impose hypotheses
is not $M$ itself, but the closed $G$-cyclic subspace
\begin{equation}
\mathcal C_G(M)
:=
\overline{\operatorname{span}}
\left\{G^{n}v:v\in M,\ n\ge0\right\} .
\label{eq:2-1}
\end{equation}
It is the smallest closed $G$-invariant subspace containing $M$ and
records the spectral directions of $G$ coupled to the prescribed task
space. The next lemma shows, in particular, that the closed span of the
unperturbed orbit over any interval of positive length is already the
whole cyclic task space; the space itself does not depend on the
comparison horizon.

\begin{lem}
\label{lem:cyclic-orbit}
For every $\sigma>0$,
\begin{equation}
\mathcal C_G(M)
=
\overline{\operatorname{span}}
\left\{S(t)v:v\in M,\ 0\le t\le\sigma\right\}.
\label{eq:cyclic-orbit}
\end{equation}
Moreover, $\mathcal C_G(M)$ is the smallest closed $G$-invariant
subspace containing $M$. Since $G$ is self-adjoint, this subspace reduces
$G$; in particular, it is invariant under $S(t)$ and $G^{1/2}$, and
\begin{equation}
\left\|G^{1/2}|_{\mathcal C_G(M)}\right\|^{2}
=
\left\|G|_{\mathcal C_G(M)}\right\|.
\label{eq:cyclic-half-norm}
\end{equation}
\end{lem}

\begin{proof}
Let $N:=\mathcal C_G(M)$. Since $G$ is bounded, the exponential series
shows that $S(t)M\subseteq N$ for every $t\ge0$. This proves that the
right-hand side of \prettyref{eq:cyclic-orbit} is contained in $N$.

Conversely, fix $v\in M$ and $n\ge1$. For $h>0$,
\[
\left(\frac{I-e^{-hG}}{h}\right)^{n}v
=
h^{-n}\sum_{k=0}^{n}(-1)^{k}\binom{n}{k}e^{-khG}v.
\]
If $h\le\sigma/n$, every time $kh$ belongs to $[0,\sigma]$. Moreover,
$(I-e^{-hG})/h\to G$ in operator norm as $h\downarrow0$, and hence the
left-hand side converges to $G^{n}v$. Thus every $G^{n}v$ belongs to the
right-hand side of \prettyref{eq:cyclic-orbit}; the case $n=0$ follows
from $S(0)v=v$. Taking closed linear spans gives the reverse inclusion.

Definition \prettyref{eq:2-1} immediately gives the minimality of
$N$ among closed $G$-invariant subspaces containing $M$. A closed
invariant subspace of the bounded self-adjoint operator $G$ is reducing,
so the functional calculus shows that $S(t)$ and $G^{1/2}$ preserve
$N$. Finally, the spectral theorem for the positive operator $G|_N$
gives \prettyref{eq:cyclic-half-norm}.
\end{proof}

The characterization in \prettyref{lem:cyclic-orbit} identifies
$\mathcal C_G(M)$ as the closed span of all output directions reached
from $M$ by the autonomous flow over any interval of positive length.
It is therefore the natural space on which to impose the task-capture
hypothesis used in \prettyref{sec:3}:
\[
\left\Vert \left(I-P\right)T^{*}u\right\Vert
\le\epsilon\left\Vert T^{*}u\right\Vert ,
\qquad u\in\mathcal C_G(M).
\]
For each task-coupled output direction $u$, this condition says that $T^{*}u$,
the adjoint of the Jacobian, lies predominantly in the active parameter
subspace $\operatorname{ran}P$. Equivalently, its component in the
discarded parameter directions has norm at most an $\epsilon$-fraction
of the full tangent norm, and hence carries at most an $\epsilon^{2}$-fraction
of the corresponding tangent energy.

Under this hypothesis, the estimates in \prettyref{sec:3} show that,
for initial data in $M$, the residual evolution of the restricted system
remains close to that of the full system on the interval $[0,\tau]$.
By contrast, \prettyref{sec:4} gives a complementary finite-dimensional
comparison that does not assume the task-capture condition. Instead, it
uses invariance of a common task space and decomposes the perturbation
into attenuation within task-coupled spectral modes and mixing between
distinct modes.

The horizon independence established in
\prettyref{lem:cyclic-orbit} follows from the autonomous structure:
every orbit segment is generated by the same operator $G$, and finite
differences formed from an arbitrarily short positive interval recover
each $G^{n}v$, $v\in M$, in the closed linear span of the orbit.
For a time-dependent generator, the analogous horizon independence need
not hold. Later stages of the evolution may couple the task to directions
not reached at earlier times, so the closed orbit span may genuinely
depend on the chosen horizon. This distinction motivates the
finite-horizon task-space construction introduced in
\prettyref{sec:na-ntk}.

We first develop the autonomous finite-time
estimates in the next two sections.

\section{Finite-time perturbation estimates}\label{sec:3}

Our first main result is Theorem~\ref{thm:2-2}. As we noted before, our fixed-kernel linearization tool and our NTK-analysis will take the form of pairs of semigroups of operators on systems of Hilbert spaces, their infinitesimal generators, and associated perturbations. This in turn allows us to obtain an a priori norm estimate \prettyref{eq:2-4} in Theorem~\ref{thm:2-2}, using Duhamel's formula for pairs of semigroups.
Let $H$ and $K$ be Hilbert spaces, and let $T:H\to K$ be a bounded operator.
Define
\[
G:=TT^{*},\quad G_{P}:=TPT^{*},\quad
Q_{P}:=G-G_{P}=T(I-P)T^{*},
\]
where $P$ is an orthogonal projection on $H$. Let
\[
S(t):=e^{-tG},\quad S_{P}(t):=e^{-tG_{P}}
\]
for $t\ge0$. Fix a closed subspace $M\subseteq K$ and $\tau>0$.
Set $N:=\mathcal C_G(M)$, with $\mathcal C_G(M)$ as in
\prettyref{eq:2-1}.
Assume that there exists $\epsilon\in\left[0,1\right)$ such that
\begin{equation}
\left\Vert \left(I-P\right)T^{*}u\right\Vert _{H}\le\epsilon\left\Vert T^{*}u\right\Vert _{H},\quad\forall u\in N, \label{eq:2-2} 
\footnote{
The hypothesis \prettyref{eq:2-2} is imposed on the
cyclic task space $N=\mathcal C_G(M)$. Although
\prettyref{lem:cyclic-orbit} describes this space using the unpruned
orbit over any positive interval, its definition is independent of the
comparison horizon. The proof of \prettyref{thm:2-2} evaluates the
perturbation only on the individual vectors $S(s)v$, with $v\in M$ and
$0\le s\le\tau$. One may therefore use instead any closed
$G$-invariant subspace $\widetilde N\subseteq K$ such that
$M\subseteq\widetilde N$. If there exists
$\epsilon\in\left[0,1\right)$ with
\[
\left\Vert \left(I-P\right)T^{*}u\right\Vert _{H}
\le\epsilon\left\Vert T^{*}u\right\Vert _{H},
\quad\forall u\in\widetilde N,
\]
then the proof of \prettyref{thm:2-2} goes through with $N$ replaced
by $\widetilde N$, and \prettyref{eq:2-4} holds with $\lambda_N$
replaced by $\Vert G|_{\widetilde N}\Vert$. In every dimension, the
minimal closed $G$-invariant choice is $\mathcal C_G(M)$. Passing to a
larger $\widetilde N$ may be convenient, but it generally strengthens
the capture hypothesis and may increase the spectral constant. 
}
\end{equation} 

and set
\begin{equation}
\lambda_{N}:=\Vert G|_{N}\Vert.\label{eq:2-3}
\end{equation}

\begin{thm}
\label{thm:2-2}Under the above assumptions,
\begin{align}
\sup_{t\in\left[0,\tau\right]}\left\Vert \left(S_{P}\left(t\right)-S\left(t\right)\right)\big|_{M}\right\Vert  & \le\epsilon\min\left\{ \tau\sqrt{\left\Vert G-G_{P}\right\Vert \lambda_{N}},\sqrt{\tau\left\Vert G-G_{P}\right\Vert /2}\right\} .\label{eq:2-4}
\end{align}
\end{thm}

\begin{proof}
For clarity, set again
\[
Q_{P}:=G-G_{P}=T(I-P)T^{*}.
\]
Since $G\ge0$ and $G_{P}\ge0$, both $S(t)$ and $S_{P}(t)$ are
contraction semigroups. By Duhamel's formula,
\[
S_{P}\left(t\right)-S\left(t\right)
=\int^{t}_{0}S_{P}\left(t-s\right)Q_{P}S\left(s\right)ds.
\]
Hence for every $v\in M$ and $t\in[0,\tau]$,
\begin{equation}
\left\Vert \left(S_{P}\left(t\right)-S\left(t\right)\right)v\right\Vert _{K}\le\int^{t}_{0}\left\Vert Q_{P}S\left(s\right)v\right\Vert _{K}ds.\label{eq:2-5}
\end{equation}
Since $I-P$ is an orthogonal projection,
\[
Q_{P}=\big[T(I-P)\big]\big[T(I-P)\big]^{*}.
\]
Therefore, for every $u\in K$,
\begin{equation}
\left\Vert Q_{P}u\right\Vert _{K}\le\left\Vert T(I-P)\right\Vert \left\Vert \left(I-P\right)T^{*}u\right\Vert _{H}=\sqrt{\left\Vert Q_{P}\right\Vert }\left\Vert \left(I-P\right)T^{*}u\right\Vert _{H}.\label{eq:2-6}
\end{equation}
If $u\in N$, then by \prettyref{eq:2-2},
\[
\left\Vert \left(I-P\right)T^{*}u\right\Vert _{H}\le\epsilon\left\Vert T^{*}u\right\Vert _{H}.
\]
Also,
\[
\left\Vert T^{*}u\right\Vert ^{2}_{H}=\left\langle u,TT^{*}u\right\rangle _{K}=\left\langle u,Gu\right\rangle _{K}=\Vert G^{1/2}u\Vert^{2}_{K}.
\]
Thus,
\[
\left\Vert \left(I-P\right)T^{*}u\right\Vert _{H}\le\epsilon\Vert G^{1/2}u\Vert_{K}.
\]
Combining this with \prettyref{eq:2-6}, we obtain
\begin{equation}
\left\Vert Q_{P}u\right\Vert _{K}\le\epsilon\sqrt{\left\Vert Q_{P}\right\Vert }\Vert G^{1/2}u\Vert_{K}\label{eq:2-7},
\end{equation}
for all $u\in N$.
Now fix $v\in M$ and $t\in[0,\tau]$. For each $s\in[0,t]$,
$S\left(s\right)v\in\mathcal C_G(M)=N$ by
\prettyref{lem:cyclic-orbit}.
Applying \prettyref{eq:2-7} to $u=S\left(s\right)v$ therefore gives the
pointwise estimate
\[
\left\Vert Q_{P}S\left(s\right)v\right\Vert _{K}
\le\epsilon\sqrt{\left\Vert Q_{P}\right\Vert }
\Vert G^{1/2}S\left(s\right)v\Vert_{K}.
\]
Substituting this estimate into \prettyref{eq:2-5}, we obtain
\begin{equation}
\left\Vert \left(S_{P}\left(t\right)-S\left(t\right)\right)v\right\Vert _{K}\le\epsilon\sqrt{\left\Vert Q_{P}\right\Vert }\int^{t}_{0}\Vert G^{1/2}S\left(s\right)v\Vert_{K}ds.\label{eq:2-8}
\end{equation}

We now estimate the integral in \eqref{eq:2-8} in two different ways.
First, since $S\left(s\right)v\in N$,
\[
\Vert G^{1/2}S(s)v\Vert_{K}\le\Vert G^{1/2}|_{N}\Vert\left\Vert S\left(s\right)v\right\Vert _{K}\le\sqrt{\lambda_{N}}\left\Vert v\right\Vert _{K},
\]
using \prettyref{eq:cyclic-half-norm}, \prettyref{eq:2-3}, and the
fact that $S\left(s\right)$ is a contraction. Substituting into
\prettyref{eq:2-8}, we obtain
\begin{equation}
\left\Vert \left(S_{P}\left(t\right)-S\left(t\right)\right)v\right\Vert _{K}\le\epsilon t\sqrt{\left\Vert Q_{P}\right\Vert \lambda_{N}}\left\Vert v\right\Vert _{K}.\label{eq:2-9}
\end{equation}
Second, by Cauchy--Schwarz,
\[
\int^{t}_{0}\Vert G^{1/2}S(s)v\Vert_{K}ds\le\sqrt{t}\left(\int^{t}_{0}\Vert G^{1/2}S(s)v\Vert^{2}_{K}ds\right)^{1/2}.
\]
Now,
\[
\frac{d}{ds}\left\Vert S\left(s\right)v\right\Vert ^{2}_{K}=-2\left\langle S\left(s\right)v,GS\left(s\right)v\right\rangle _{K}=-2\Vert G^{1/2}S\left(s\right)v\Vert^{2}_{K},
\]
and so
\[
\int^{t}_{0}\Vert G^{1/2}S\left(s\right)v\Vert^{2}_{K}ds=\frac{1}{2}\left(\left\Vert v\right\Vert ^{2}_{K}-\left\Vert S\left(t\right)v\right\Vert ^{2}_{K}\right)\le\frac{1}{2}\left\Vert v\right\Vert ^{2}_{K}.
\]
Therefore,
\[
\int^{t}_{0}\Vert G^{1/2}S\left(s\right)v\Vert_{K}ds\le\sqrt{\frac{t}{2}}\left\Vert v\right\Vert _{K}.
\]
Substituting into \prettyref{eq:2-8}, we get
\begin{equation}
\left\Vert \left(S_{P}\left(t\right)-S\left(t\right)\right)v\right\Vert _{K}\le\epsilon\sqrt{\frac{t\left\Vert Q_{P}\right\Vert }{2}}\left\Vert v\right\Vert _{K}.\label{eq:2-10}
\end{equation}
Taking the operator norm over unit vectors $v\in M$ in \prettyref{eq:2-9}
and \prettyref{eq:2-10}, and then taking the supremum over $t\in\left[0,\tau\right]$,
gives \prettyref{eq:2-4}.
\end{proof}

\begin{rem}
\label{rem:2-2}
The projection $P$ in \prettyref{eq:2-2} models
 pruning, or freezing, of active parameter directions. Indeed, the
directions in $ran P$ are retained, while those 
in $ker P$ are completely removed. More generally, one
may replace $P$ by a positive contraction $A$ on $H$, with $0\le A\le I$,
and define
\[
G_{A}:=TAT^{*},\quad Q_{A}:=G-G_{A}=T(I-A)T^{*},\quad
S_{A}\left(t\right):=e^{-tG_{A}},\qquad t\ge0.
\]
In this form, the model allows ``soft" pruning, in the sense that
parameter directions are not only kept or removed, but may be attenuated.
The proof of Theorem~\ref{thm:2-2} still goes through when one replaces \prettyref{eq:2-2} with the following condition
\begin{equation}
\Vert\left(I-A\right)^{1/2}T^{*}u\Vert_{H}\le\epsilon\left\Vert T^{*}u\right\Vert _{H},\quad\forall u\in N,\label{eq:2-11}
\end{equation}
and $G_{P}$, $Q_{P}$, and
$S_{P}\left(t\right)$ with $G_{A}$, $Q_{A}$, and
$S_{A}\left(t\right)$, respectively. Indeed,
\[
Q_{A}=\big[T(I-A)^{1/2}\big]\big[T(I-A)^{1/2}\big]^{*},
\]
so that 
\begin{equation}
\left\langle u,\left(G-G_{A}\right)u\right\rangle _{K}\le\epsilon^{2}\left\langle u,Gu\right\rangle _{K}.\quad\forall u\in N,\label{eq:2-12}
\end{equation}
Since
\[
\left\langle u,\left(G-G_{A}\right)u\right\rangle _{K}=\left\langle T^{*}u,\left(I-A\right)T^{*}u\right\rangle _{H}=\Vert\left(I-A\right)^{1/2}T^{*}u\Vert^{2}_{H},
\]
we have
\[
\sup_{0\le t\le\tau}
\|(S_{A}(t)-S(t))|_{M}\|
\le
\epsilon\min\left\{
\tau\sqrt{\|Q_{A}\|\lambda_{N}},
\sqrt{\tau\|Q_{A}\|/2}
\right\}.
\]
\end{rem}

\begin{thm}
\label{thm:2-4} Assume that $K$ is a finite-dimensional Hilbert-space. Let
\[
G=\sum^{q}_{\alpha=1}\lambda_{\alpha}E_{\alpha}
\]
be the spectral decomposition of $G = TT^{*}$, where $\lambda_{1},\dots,\lambda_{q}$
are the distinct eigenvalues of $G$ and $E_{1},\dots,E_{q}$ are
the corresponding orthogonal spectral projections.
Let $M$ be a subspace of $K$. Then
\begin{equation}
\mathcal C_G(M)=\bigoplus^{q}_{\alpha=1}E_{\alpha}M.\label{eq:2-13}
\end{equation}
Moreover, for every $\sigma>0$, finitely many semigroup snapshots in
$[0,\sigma]$ span the entire cyclic task space. More precisely, if
$q=1$, then $\mathcal C_G(M)=S(0)M=M$, while if $q\ge2$ and
\[
t_j:=(j-1)\frac{\sigma}{q-1},\qquad j=1,\dots,q,
\]
then
\begin{equation}
\mathcal C_G(M)
=
\operatorname{span}\left\{S(t_j)M:1\le j\le q\right\}.
\label{eq:finite-snapshots}
\end{equation}
\end{thm}

\begin{proof}
Set $L:=\bigoplus_{\alpha=1}^{q}E_{\alpha}M$. For $v\in M$ and
$n\ge0$,
\[
G^nv=\sum_{\alpha=1}^{q}\lambda_\alpha^nE_\alpha v\in L,
\]
so $\mathcal C_G(M)\subseteq L$. Conversely, for each $\alpha$,
\[
E_\alpha
=
\prod_{\substack{1\le\beta\le q\\ \beta\ne\alpha}}
\frac{G-\lambda_\beta I}{\lambda_\alpha-\lambda_\beta}
\]
is a polynomial in $G$. Hence $E_\alpha M\subseteq\mathcal C_G(M)$,
which proves \prettyref{eq:2-13}.

It remains to verify the finite-snapshot statement. The case $q=1$
is immediate. Suppose that $q\ge2$ and let the $t_j$ be as in the
statement. For $v\in M$,
\[
S\left(t_{j}\right)v=\sum^{q}_{\beta=1}e^{-t_{j}\lambda_{\beta}}E_{\beta}v.
\]
Set $r_{\beta}:=e^{-\frac{\sigma}{q-1}\lambda_{\beta}}$. Since the
$\lambda_{\beta}$'s are are distinct, the numbers $r_{1},\dots,r_{q}$
are also distinct. Moreover, $e^{-t_{j}\lambda_{\beta}}=r^{j-1}_{\beta}$.
Hence,
\[
S\left(t_{j}\right)v=\sum^{q}_{\beta=1}r^{j-1}_{\beta}E_{\beta}v, \qquad \text{for }j=1,\dots,q.
\]
Let $V:=(r^{j-1}_{\beta})^{q}_{j,\beta=1}$. This is a Vandermonde
matrix. Since the numbers $r_{1},\dots,r_{q}$ are distinct, $V$
is invertible. Therefore the vectors $E_{1}v,\dots,E_{q}v$ are linear
combinations of the vectors $S\left(t_{1}\right)v,\dots,S\left(t_{q}\right)v$.
Thus every $E_\alpha v$ belongs to
$\operatorname{span}\{S(t_j)v:1\le j\le q\}$. By
\prettyref{eq:2-13}, this proves that $\mathcal C_G(M)$ is contained in
the right-hand side of \prettyref{eq:finite-snapshots}. The reverse
inclusion follows from \prettyref{lem:cyclic-orbit}.
\end{proof}

\begin{rem}
The estimate in \prettyref{thm:2-2} is most informative when the
$G$-invariant subspace generated by $M$ is a proper subspace of
$K$. Indeed, if $M=K$, then \prettyref{thm:2-4} gives
$\mathcal C_G(M)=K$,
so the hypothesis \prettyref{eq:2-2} must hold on all of $K$. In
that case, the theorem becomes a global perturbation statement. Thus
the point of the result is that one only needs control on the spectral
directions of $G$ that are coupled to $M$. In particular, a large
kernel of $G$ does not by itself make the cyclic space large; the zero
eigenspace contributes only to the projection of $M$ onto that eigenspace.
\end{rem}

\begin{cor}
\label{cor:2-6} Under the hypotheses of \prettyref{thm:2-4}, set
$N:=\mathcal C_G(M)$. Then
\begin{equation}
\lambda_N=\Vert G|_{N}\Vert
=\max\left\{ \lambda_{\alpha}:E_{\alpha}M\ne\left\{ 0\right\} \right\}.
\label{eq:2-14}
\end{equation}
with the convention that the maximum is $0$ when $M=\{0\}$.
\end{cor}

\begin{proof}
By \prettyref{thm:2-4}, $N=\bigoplus^{q}_{\alpha=1}E_{\alpha}M$,
and each nonzero subspace $E_{\alpha}M$ is contained in the eigenspace
$ran E_{\alpha}$, on which $G$ acts as multiplication
by $\lambda_{\alpha}$. Hence, the spectrum of $G|_{N}$ consists
of those $\lambda_{\alpha}$ for which $E_{\alpha}M\ne\left\{ 0\right\} $,
and the norm of this positive operator is their maximum. If $M=\{0\}$,
then $N=\{0\}$ and $\lambda_N=0$.
\end{proof}

\section{Refinement on the task space}\label{sec:4}

We now work in the finite-dimensional setting and split the perturbation
into two parts: the part that stays diagonal with respect to the spectral
decomposition of $G$, and the part that mixes different spectral
modes. This gives a complementary comparison on the minimal
$G$-invariant subspace generated by $M$. Unlike \prettyref{thm:2-2},
the results below do not assume the relative task-capture condition
\prettyref{eq:2-2}; they instead use the invariance
\prettyref{eq:3-2} and quantify diagonal attenuation and off-diagonal
mixing through $\delta$ and $\rho$.
Throughout this section, assume that $K$ is finite-dimensional and that $M\ne\{0\}$. Let
\[
N:=\mathcal C_G(M).
\]
The cyclic task space $N$ is independent of the comparison horizon;
every occurrence of $\tau$ below refers only to the interval on which
the estimate is evaluated.
By \prettyref{thm:2-4},
\[
N=\bigoplus_{\alpha}E_{\alpha}M,
\]
and by \prettyref{cor:2-6},
\begin{equation}
\lambda_{N}:=\Vert G|_{N}\Vert=\max\left\{ \lambda_{\alpha}:E_{\alpha}M\ne\left\{ 0\right\} \right\} .\label{eq:3-1}
\end{equation}
Assume also that
\begin{equation}
G_{P}N\subseteq N.\label{eq:3-2}
\end{equation}
Under this invariance assumption, write
\begin{equation}
D:=\sum_{\alpha}E_{\alpha}G_{P}E_{\alpha}|_{N},\quad R:=\sum_{\alpha\ne\beta}E_{\alpha}G_{P}E_{\beta}|_{N}.\label{eq:3-3}
\end{equation}
Then
\begin{equation}
G_{P}|_{N}=D+R.\label{eq:3-4}
\end{equation}
The operator $D$ is the block diagonal part of $G_{P}$ relative
to the spectral decomposition of $G|_{N}$, while $R$ is the off-diagonal
mixing part.

\begin{thm}
\label{thm:3-1} Suppose that there exists $\delta\in\left[0,1\right]$ such that
\begin{equation}
0\le G|_{N}-D\le\delta G|_{N}.\label{eq:3-5}
\end{equation}
Set
\begin{equation}
\rho:=\left\Vert R\right\Vert .\label{eq:3-6}
\end{equation}
Then for every $t\ge0$,
\begin{equation}
\left\Vert \left(S_{P}\left(t\right)-S\left(t\right)\right)|_{M}\right\Vert \le1-e^{-t\delta\lambda_{N}}+t\rho.\label{eq:3-7}
\end{equation}
Since the right-hand side of \prettyref{eq:3-7} is nondecreasing in
$t$, this estimate also holds uniformly on every interval $[0,\tau]$,
with the right-hand side evaluated at $t=\tau$.
Since $1-e^{-x}\le x$ for $x\ge0$, \prettyref{eq:3-7} also yields
\begin{equation}
\left\Vert \left(S_{P}\left(t\right)-S\left(t\right)\right)|_{M}\right\Vert \le t\left(\delta\lambda_{N}+\rho\right),\quad t\ge0.\label{eq:3-9}
\end{equation}
\end{thm}

\begin{proof}
Since $M\subseteq N$, it is enough to estimate the difference on
$N$. By \prettyref{eq:3-2}, the operator $G_{P}|_{N}$ is well defined,
and by \prettyref{eq:3-4},
\begin{equation}
\left\Vert \left(S_{P}\left(t\right)-S\left(t\right)\right)|_{M}\right\Vert \le\Vert e^{-tG_{P}|_{N}}-e^{-tD}\Vert+\Vert e^{-tD}-e^{-tG|_{N}}\Vert.\label{eq:3-10}
\end{equation}
We first estimate the off-diagonal part. Since $G_{P}|_{N}\ge0$ and
$D\ge0$, both $e^{-tG_{P}|_{N}}$ and $e^{-tD}$ are contraction
semigroups on $N$. By Duhamel's formula,
\[
e^{-tG_{P}|_{N}}-e^{-tD}=-\int^{t}_{0}e^{-\left(t-s\right)G_{P}|_{N}}Re^{-sD}ds.
\]
Therefore,
\[
\Vert e^{-tG_{P}|_{N}}-e^{-tD}\Vert\le\int^{t}_{0}\Vert e^{-\left(t-s\right)G_{P}|_{N}}\Vert\left\Vert R\right\Vert \Vert e^{-sD}\Vert ds\le t\rho.
\]
Thus,
\begin{equation}
\Vert e^{-tG_{P}|_{N}}-e^{-tD}\Vert\le t\rho.\label{eq:3-11}
\end{equation}

We next estimate the diagonal part. By \prettyref{eq:3-3}, each spectral
subspace $E_{\alpha}N$ is invariant under $D$, and therefore $D$
commutes with every $E_{\alpha}\big|_{N}$. Since
\[
G|_{N}=\sum_{\alpha}\lambda_{\alpha}E_{\alpha}|_{N},
\]
it follows that
\begin{equation}
DG|_{N}=G|_{N}D.\label{eq:3-12}
\end{equation}
As $G|_{N}$ and $D$ are commuting self-adjoint operators on the finite-dimensional
space $N$, there exists an orthonormal basis of $N$ consisting of
common eigenvectors. Thus, we may write
\[
G|_{N}e_{j}=\lambda_{j}e_{j},\quad De_{j}=\mu_{j}e_{j},
\]
for some $\lambda_{j},\mu_{j}\ge0$. 

By \prettyref{eq:3-5},
\[
0\le\lambda_{j}-\mu_{j}\le\delta\lambda_{j}\le\delta\lambda_{N},
\]
for every $j$, where $\lambda_{N}$ is given by \prettyref{eq:3-1}.
Hence,
\[
\left|e^{-t\mu_{j}}-e^{-t\lambda_{j}}\right|=e^{-t\mu_{j}}\left(1-e^{-t\left(\lambda_{j}-\mu_{j}\right)}\right)\le1-e^{-t\left(\lambda_{j}-\mu_{j}\right)}\le1-e^{-t\delta\lambda_{N}}.
\]
Taking the supremum over $j$, we obtain
\begin{equation}
\Vert e^{-tD}-e^{-tG|_{N}}\Vert\le1-e^{-t\delta\lambda_{N}},\label{eq:3-13}
\end{equation}
and combining \prettyref{eq:3-10}, \prettyref{eq:3-11}, and \prettyref{eq:3-13}
gives \prettyref{eq:3-7}. Since its right-hand side is nondecreasing in
$t$, the stated uniform consequence follows by taking the supremum over
$t\in[0,\tau]$. Finally, \prettyref{eq:3-9} follows from
the elementary inequality $1-e^{-x}\le x$ for $x\ge0$.
\end{proof}

\begin{rem}
\label{rem:3-2} The estimate in \prettyref{thm:3-1} separates the two
effects of pruning on the task space $N$. The term $1-e^{-t\delta\lambda_{N}}$
comes from the block diagonal part $D$ and measures the attenuation along
the spectral modes of $G$. The term $t\rho$ comes from the off-diagonal
part $R$ and measures the mixing between distinct $G$-eigenspaces. In
particular, if $R=0$, then $G_{P}|_{N}$ commutes with $G|_{N}$,
and \prettyref{thm:3-1} reduces to a modewise comparison. Under
\prettyref{eq:3-2}, condition \prettyref{eq:3-5} is always available
with $\delta=1$, since
\[
G|_{N}-D=\sum_{\alpha}E_{\alpha}Q_{P}E_{\alpha}|_{N},
\qquad 0\le Q_{P}\le G.
\]
Thus a value $\delta<1$ supplies additional quantitative information
rather than an existence condition.
\end{rem}

\begin{cor}
\label{cor:3-3} Under the hypotheses of \prettyref{thm:3-1}, if
$R=0$, then
\begin{equation}
\left\Vert \left(S_{P}\left(t\right)-S\left(t\right)\right)|_{M}\right\Vert \le1-e^{-t\delta\lambda_{N}},\quad t\ge0.\label{eq:3-14}
\end{equation}
Since the right-hand side is nondecreasing in $t$, the corresponding
uniform estimate on every interval $[0,\tau]$ follows by evaluating it
at $t=\tau$.
\end{cor}

We now refine the mixing term in \prettyref{thm:3-1}.
Let $A:=\left\{ \alpha:E_{\alpha}M\neq\left\{ 0\right\} \right\} $ and
$N=\mathcal C_G(M)=\bigoplus_{\alpha\in A}E_{\alpha}M$. For
$\alpha,\beta\in A$, $\alpha\ne\beta$,
set
\[
R_{\alpha\beta}:=E_{\alpha}G_{P}E_{\beta}|_{E_{\beta}M}.
\]
Define
\begin{equation}
\rho_{\mathrm{row}}:=\max_{\alpha\in A}\sum_{\beta\in A,\beta\ne\alpha}\left\Vert R_{\alpha\beta}\right\Vert ,\quad\rho_{\mathrm{col}}:=\max_{\beta\in A}\sum_{\alpha\in A,\alpha\ne\beta}\left\Vert R_{\alpha\beta}\right\Vert .\label{eq:3-18}
\end{equation}
Further set
\begin{equation}
\rho_{\mathrm{blk}}:=\sqrt{\rho_{\mathrm{row}}\rho_{\mathrm{col}}}.\label{eq:3-19}
\end{equation}

\begin{thm}
\label{thm:3-6} Under the hypotheses of \prettyref{thm:3-1},
\begin{equation}
\left\Vert R\right\Vert \le\rho_{\mathrm{blk}}.\label{eq:3-20}
\end{equation}
Consequently, for every $t\ge0$,
\begin{equation}
\left\Vert \left(S_{P}\left(t\right)-S\left(t\right)\right)|_{M}\right\Vert \le1-e^{-t\delta\lambda_{N}}+t\rho_{\mathrm{blk}}.\label{eq:3-21}
\end{equation}
Since the right-hand side is nondecreasing in $t$, the corresponding
uniform estimate on every interval $[0,\tau]$ follows by evaluating it
at $t=\tau$.
\end{thm}

\begin{proof}
Write $x=\sum_{\beta\in A}x_{\beta}$, $y=\sum_{\alpha\in A}y_{\alpha}$ for $x,y\in N$,
where $x_{\beta}:=E_{\beta}x$, $y_{\alpha}:=E_{\alpha}y$. Since
$R$ has vanishing diagonal blocks,
\[
\left\langle Rx,y\right\rangle =\sum_{\alpha\in A}\sum_{\beta\in A,\beta\ne\alpha}\left\langle R_{\alpha\beta}x_{\beta},y_{\alpha}\right\rangle .
\]
Hence,
\[
\left|\left\langle Rx,y\right\rangle \right|\le\sum_{\alpha\in A}\sum_{\beta\in A,\beta\ne\alpha}\left\Vert R_{\alpha\beta}\right\Vert \left\Vert x_{\beta}\right\Vert \left\Vert y_{\alpha}\right\Vert .
\]
Set
\[
a_{\alpha\beta}:=\begin{cases}
\left\Vert R_{\alpha\beta}\right\Vert , & \alpha\ne\beta,\\
0, & \alpha=\beta.
\end{cases}
\]
By the Schur test for scalar matrices,
\[
\left\Vert \left(a_{\alpha\beta}\right)\right\Vert _{\ell^{2}\to\ell^{2}}\le\sqrt{\rho_{\mathrm{row}}\rho_{\mathrm{col}}}=\rho_{\mathrm{blk}}.
\]
Therefore,
\[
\left|\left\langle Rx,y\right\rangle \right| \le \sum\nolimits_{\alpha,\beta}a_{\alpha\beta}\left\Vert x_{\beta}\right\Vert \left\Vert y_{\alpha}\right\Vert \le\rho_{\mathrm{blk}}\left(\sum\nolimits_{\beta}\left\Vert x_{\beta}\right\Vert ^{2}\right)^{1/2}\left(\sum\nolimits_{\alpha}\left\Vert y_{\alpha}\right\Vert ^{2}\right)^{1/2}=\rho_{\mathrm{blk}}\left\Vert x\right\Vert \left\Vert y\right\Vert .
\]
Taking the supremum over unit vectors $x,y\in N$ gives \prettyref{eq:3-20}.
The bound \prettyref{eq:3-21} and its stated uniform consequence now
follow from \prettyref{thm:3-1}.
\end{proof}

\begin{rem}
\label{rem:3-7} The quantity $\rho_{\mathrm{blk}}$ resolves the
mixing term blockwise along the spectral decomposition of $G$. In
particular, \prettyref{thm:3-6} shows that the contribution of $R$
is small if, for each task-coupled $G$-mode, the total coupling to the
other task-coupled modes is small both row-wise and column-wise.
\end{rem}

\begin{cor}
\label{cor:3-8} Under the hypotheses of \prettyref{thm:3-1}, if
there exists $\eta\ge0$ such that
\begin{equation}
\sum_{\beta\in A,\beta\ne\alpha}\left\Vert E_{\alpha}G_{P}E_{\beta}|_{E_{\beta}M}\right\Vert \le\eta,\quad\alpha\in A,\label{eq:3-23}
\end{equation}
and
\begin{equation}
\sum_{\alpha\in A,\alpha\ne\beta}\left\Vert E_{\alpha}G_{P}E_{\beta}|_{E_{\beta}M}\right\Vert \le\eta,\quad\beta\in A,\label{eq:3-24}
\end{equation}
then
\begin{equation}
\left\Vert \left(S_{P}\left(t\right)-S\left(t\right)\right)|_{M}\right\Vert \le1-e^{-t\delta\lambda_{N}}+t\eta,\quad t\ge0.\label{eq:3-25}
\end{equation}
\end{cor}

\begin{proof}
Under \prettyref{eq:3-23}--\prettyref{eq:3-24}, we have $\rho_{\mathrm{row}}\le\eta$,
$\rho_{\mathrm{col}}\le\eta$, hence $\rho_{\mathrm{blk}}\le\eta$.
The conclusion follows from \prettyref{thm:3-6}.
\end{proof}

\section{Ces\`aro averages and ergodic comparison on the task space}
\label{sec:ergodic-task}

The estimates in Sections \ref{sec:3} and \ref{sec:4} control the
difference of the two semigroups at each time in a fixed interval. We
now pass to Ces\`aro averages. Set
\[
C_{L}:=\frac{1}{L}\int_{0}^{L}S(t)\,dt,
\qquad
C_{P,L}:=\frac{1}{L}\int_{0}^{L}S_{P}(t)\,dt,
\qquad L>0,
\]
where, as before, $S(t)=e^{-tG}$ and $S_{P}(t)=e^{-tG_{P}}$. Averages of
this kind are used here to quantify how well $C_{P,L}|_{M}$ approximates
$C_{L}|_{M}$ as $L$ grows and to identify the stationary component that
may remain in their difference. Such averages are classical in the mean ergodic theory of
$C_{0}$-semigroups; see, for example,
\cite{kido1984mean,albanese2012mean,barki2021uniform,gomilko2012bernstein}.

In our setting, the generator $G=TT^{*}$ is bounded, self-adjoint, and
positive, so the relevant ergodic limits can be derived directly by the
spectral calculus.
Throughout this section, $N\subseteq K$ denotes a closed $G$-invariant
subspace with $M\subseteq N$. The minimal such space is
$\mathcal C_G(M)$; retaining a general $N$ is convenient when a larger
common space is required for the pruned dynamics. The first results below concern averaged
tangent leakage; they depend only on the unpruned semigroup and do not
require the invariance hypothesis \prettyref{eq:3-2}. The comparison of
the averages $C_{P,L}$ and $C_{L}$ themselves does require
\prettyref{eq:3-2}, and is taken up in the second half of the section.
For exact hard truncation, the relative hypothesis \prettyref{eq:2-2}
may fail even on $\mathcal C_G(M)$, and hence may also fail on a larger
invariant space $N$, since the discarded tangent component need not be
uniformly small relative to \(\|T^{*}u\|_{H}\). We therefore work
with the following weaker hypothesis: We say that an orthogonal projector $P$ on $H$ satisfies the
\emph{mixed task-tail condition} on $N$ with constants
\((\alpha,\beta)\in[0,\infty)^{2}\) if
\begin{equation}
\label{eq:ergodic-mixed-def}
\|(I-P)T^{*}u\|_{H}^{2}
\le
\alpha\|T^{*}u\|_{H}^{2}+\beta\|u\|_{K}^{2},
\qquad u\in N.
\end{equation}
Equivalently,
\begin{equation}
\label{eq:ergodic-mixed-form}
\langle u,(G-G_{P})u\rangle_{K}
\le
\alpha\langle u,Gu\rangle_{K}+\beta\|u\|_{K}^{2},
\qquad u\in N.
\end{equation}
In contrast with \prettyref{eq:2-2}, neither $\alpha$ nor $\beta$ is
assumed to be small unless explicitly stated. The constant $\alpha$
measures the relative part of the discarded tangent energy, while
$\beta$ records an additive task-space error. If the task space has a
lower spectral edge
\[
G|_{N}\ge \mu_{N}I_{N}
\qquad (\mu_{N}>0),
\]
then
\[
\|T^{*}u\|_{H}^{2}=\langle u,Gu\rangle_{K}\ge \mu_{N}\|u\|_{K}^{2},
\qquad u\in N,
\]
and \prettyref{eq:ergodic-mixed-def} reduces to the effective pure-relative bound
\begin{equation}
\label{eq:ergodic-mixed-eff}
\|(I-P)T^{*}u\|_{H}^{2}
\le
\left(\alpha+\frac{\beta}{\mu_{N}}\right)\|T^{*}u\|_{H}^{2},
\qquad u\in N.
\end{equation}
Thus the additive term can be absorbed only when a positive lower task-energy
bound is available.

\begin{rem}[Motivation for the additive task-space tolerance]
\label{rem:mixed-task-tail-interpretation}
The term \(\beta\|u\|_{K}^{2}\) in the mixed task-tail condition \eqref{eq:ergodic-mixed-def} allows a two-term a priori estimate for hard
truncations in directions, where the \(G\)-energy
\(\langle u,Gu\rangle_{K}=\|T^{*}u\|_{H}^{2}\) does not uniformly
control the ambient task norm.

This contraction type term plus an  additive structure is analogous in spirit to the
restricted approximate invertibility condition (RAIC) in one-bit compressed
sensing, where a pure restricted isometry property (RIP)-type condition is replaced by a local estimate
on annuli with an additive error term
\cite{friedlander2021nbiht, chen2024optimal, abdalla2026robust, chen2026one, matsumoto2024binary, matsumoto2024robust}. 
It is also consistent with the
role of penalty or regularizing terms in variational, iterative 
estimation schemes, where a primary residual or data term is
supplemented by a stabilizing term controlling components; see, for example,
\cite{MR5062983,MR4831876,MR4770680}.  The usefulness of a possibly
non-small \(\beta\) is determined by the subsequent averaged estimates:
its contribution is propagated explicitly through quantities such as
\[
\frac{\beta}{L}\int_{0}^{L}\|S(t)v\|_{K}^{2}\,dt,
\]
and is effective precisely when this averaged task energy decays.
\end{rem}

We first record the frozen Ces\`aro limit, which depends only on the
unpruned semigroup.

\begin{thm}
\label{thm:ergodic-frozen}
Let
\[
C_{L}=\frac{1}{L}\int_{0}^{L}e^{-tG}\,dt,
\qquad G=TT^{*}.
\]
Then
\[
C_{L}\to P_{\ker G}=P_{\ker T^{*}}
\qquad\text{strongly on }K, \text{ as }L\to\infty.
\]
Consequently, for every \(v\in K\),
\[
(I-P)T^{*}C_{L}v\to0
\qquad\text{in }H, \text{ as }L\to\infty.
\]
More generally, if \(Y\) is a Hilbert space and \(U\in\mathcal B(K,Y)\) satisfies
\(U P_{\ker G}=0\), then
\[
U C_{L}v\to0
\qquad\text{for every }v\in K.
\]
\end{thm}

\begin{proof}
Define
\[
\psi_{L}(\lambda):=
\begin{cases}
\dfrac{1-e^{-L\lambda}}{L\lambda}, & \lambda>0,\\[1ex]
1, & \lambda=0.
\end{cases}
\]
Since \(G\) is bounded, self-adjoint, and positive, the functional calculus gives
\[
C_{L}=\frac{1}{L}\int_{0}^{L}e^{-tG}\,dt=\psi_{L}(G).
\]
Also, \(0\le \psi_{L}(\lambda)\le1\) for all \(\lambda\ge0\), and
\(\psi_{L}(\lambda)\to 0\) for every \(\lambda>0\) while
\(\psi_{L}(0)=1\). Hence, for every \(v\in K\), the spectral theorem yields
\[
\|C_{L}v-P_{\ker G}v\|^{2}
=
\int_{[0,\|G\|]}
\bigl|\psi_{L}(\lambda)-\mathbf 1_{\{0\}}(\lambda)\bigr|^{2}
\,d\mu_{v}(\lambda)
\longrightarrow0,
\]
where \(\mu_{v}\) is the spectral measure of \(G\) associated with \(v\).
Thus \(C_{L}\to P_{\ker G}\) strongly.
Since
\[
\ker G
=
\{u\in K:\langle u,Gu\rangle=0\}
=
\{u\in K:\|T^{*}u\|^{2}=0\}
=
\ker T^{*},
\]
\[
(I-P)T^{*}P_{\ker G}=0.
\]
The second statement now follows from the strong convergence of \(C_{L}\) and the
boundedness of \((I-P)T^{*}\). The last assertion is proved in exactly the same
way.
\end{proof}

The next estimate applies the mixed condition
\prettyref{eq:ergodic-mixed-def} along the unpruned orbit. It gives an
averaged leakage bound which does not assume that $N$ is invariant under
$G_{P}$.

\begin{cor}
\label{cor:ergodic-mixed}
Let \(N\subseteq K\) be a closed \(G\)-invariant subspace with \(M\subseteq N\).
Assume that \prettyref{eq:ergodic-mixed-def} holds on \(N\). Then for every
\(v\in M\) and every \(L>0\),
\begin{equation}
\label{eq:ergodic-mixed-2}
\|(I-P)T^{*}C_{L}v\|_{H}^{2}
\le
\frac{\alpha}{2L}\|v\|_{K}^{2}
+
\frac{\beta}{L}\int_{0}^{L}\|S(t)v\|_{K}^{2}\,dt.
\end{equation}
If, in addition, there exists \(\mu_{N}>0\) such that
\[
G|_{N}\ge \mu_{N}I_{N},
\]
then
\begin{equation}
\label{eq:ergodic-mixed-3}
\|(I-P)T^{*}C_{L}|_{M}\|^{2}
\le
\frac{\alpha}{2L}
+
\frac{\beta}{2\mu_{N}L}.
\end{equation}
In particular,
\[
\|(I-P)T^{*}C_{L}|_{M}\|=O(L^{-1/2})
\qquad (L\to\infty).
\]
\end{cor}

\begin{proof}
Since \(C_{L}v=\frac{1}{L}\int_{0}^{L}S(t)v\,dt\), Jensen's inequality gives
\[
\|(I-P)T^{*}C_{L}v\|_{H}^{2}
\le
\frac{1}{L}\int_{0}^{L}\|(I-P)T^{*}S(t)v\|_{H}^{2}\,dt.
\]
Applying \prettyref{eq:ergodic-mixed-def} with \(u=S(t)v\in N\), we get
\[
\|(I-P)T^{*}C_{L}v\|_{H}^{2}
\le
\frac{\alpha}{L}\int_{0}^{L}\|T^{*}S(t)v\|_{H}^{2}\,dt
+
\frac{\beta}{L}\int_{0}^{L}\|S(t)v\|_{K}^{2}\,dt.
\]
Now
\[
\|T^{*}S(t)v\|_{H}^{2}
=
\langle S(t)v,GS(t)v\rangle_{K}
=
\|G^{1/2}S(t)v\|_{K}^{2},
\]
and, as in the proof of \prettyref{thm:2-2},
\[
\int_{0}^{L}\|G^{1/2}S(t)v\|_{K}^{2}\,dt
=
\frac{1}{2}\bigl(\|v\|_{K}^{2}-\|S(L)v\|_{K}^{2}\bigr)
\le
\frac{1}{2}\|v\|_{K}^{2}.
\]
This proves \prettyref{eq:ergodic-mixed-2}.
If \(G|_{N}\ge \mu_{N}I_{N}\), then
\[
\|S(t)|_{N}\|\le e^{-\mu_{N}t},
\]
so for \(v\in M\subseteq N\),
\[
\frac{1}{L}\int_{0}^{L}\|S(t)v\|_{K}^{2}\,dt
\le
\frac{1}{L}\int_{0}^{L}e^{-2\mu_{N}t}\,dt\,\|v\|_{K}^{2}
\le
\frac{1}{2\mu_{N}L}\|v\|_{K}^{2}.
\]
Substituting this into \prettyref{eq:ergodic-mixed-2} gives
\prettyref{eq:ergodic-mixed-3}.
\end{proof}

We now turn to the averaged difference of the two semigroups. From this
point on we return to the finite-dimensional task-space setting of
\prettyref{sec:4}. The pruned semigroup itself now appears, so
assumption \prettyref{eq:3-2} is needed to make \(G_{P}|_{N}\) a
well-defined positive operator on the common task space \(N\). Thus
\(N=\mathcal C_G(M)\), \(G_{P}N\subseteq N\), and
\(G_{P}|_{N}=D+R\) as in
\prettyref{eq:3-3}--\prettyref{eq:3-4}. The next statement identifies
the exact Ces\`aro limit of the averaged semigroup difference on the
task space.

\begin{thm}
\label{thm:ergodic-raw}
Assume that \(K\) is finite-dimensional and that \(G_{P}N\subseteq N\), with
\(N=\mathcal C_G(M)\) as in \prettyref{sec:4}. Define
\begin{equation}
\label{eq:ergodic-raw-1}
\Delta_{L}:=\frac{1}{L}\int_{0}^{L}(S_{P}(t)-S(t))|_{M}\,dt,
\qquad L>0.
\end{equation}
Then
\begin{equation}
\label{eq:ergodic-raw-2}
\Delta_{L}
\longrightarrow
\bigl(P_{\ker(G_{P}|_{N})}-P_{\ker(G|_{N})}\bigr)|_{M}
\qquad\text{in operator norm as }L\to\infty.
\end{equation}
Moreover,
\begin{equation}
\label{eq:ergodic-raw-3}
\ker(G|_{N})\subseteq\ker(G_{P}|_{N}).
\end{equation}
Hence, if
\begin{equation}
\label{eq:ergodic-raw-4}
E:=\ker(G_{P}|_{N})\ominus\ker(G|_{N}),
\end{equation}
then
\begin{equation}
\label{eq:ergodic-raw-5}
\Delta_{L}\to P_{E}|_{M}
\qquad\text{in operator norm as }L\to\infty.
\end{equation}
\end{thm}

\begin{proof}
Set \(B:=G|_{N}\) and \(B_{P}:=G_{P}|_{N}\). Since \(N\) is
finite-dimensional and invariant under the self-adjoint operators \(G\)
and \(G_{P}\), both \(B\) and \(B_{P}\) are positive self-adjoint
operators on \(N\). Therefore,
\[
\frac{1}{L}\int_{0}^{L}e^{-tB_{P}}\,dt=\psi_{L}(B_{P}),
\qquad
\frac{1}{L}\int_{0}^{L}e^{-tB}\,dt=\psi_{L}(B),
\]
with the same scalar function \(\psi_{L}\) as in the proof of
\prettyref{thm:ergodic-frozen}. Since \(N\) is finite-dimensional, the
convergence \(\psi_{L}(X)\to P_{\ker X}\) holds in operator norm for
every positive self-adjoint operator \(X\) on \(N\). Hence
\[
\frac{1}{L}\int_{0}^{L}e^{-tB_{P}}\,dt\to P_{\ker B_{P}},
\qquad
\frac{1}{L}\int_{0}^{L}e^{-tB}\,dt\to P_{\ker B}
\]
in operator norm on \(N\).
Since \(M\subseteq N\), restricting the difference to \(M\) yields
\prettyref{eq:ergodic-raw-2}.
To prove \prettyref{eq:ergodic-raw-3}, let \(u\in\ker(G|_{N})\). Then
\[
0=\langle u,Gu\rangle_{K}=\|T^{*}u\|_{H}^{2},
\]
so \(T^{*}u=0\). Therefore
\[
G_{P}u=TPT^{*}u=0,
\]
and hence \(u\in\ker(G_{P}|_{N})\). This proves
\prettyref{eq:ergodic-raw-3}. The last statement follows immediately from the
orthogonal decomposition
\[
\ker(G_{P}|_{N})=\ker(G|_{N})\oplus E.
\]
\end{proof}

\begin{cor}
\label{cor:ergodic-zero}
Under the hypotheses of \prettyref{thm:ergodic-raw}, the following are equivalent:
\begin{enumerate}[(i)]
\item
\(\Delta_{L}\to0\) in operator norm on \(M\) as \(L\to\infty\).
\item
\[
P_{\ker(G_{P}|_{N})}|_{M}=P_{\ker(G|_{N})}|_{M}.
\]
\item
\[
P_{E}|_{M}=0,
\]
where \(E\) is given by \prettyref{eq:ergodic-raw-4}.
\end{enumerate}
\end{cor}

Thus, the Ces\`aro average of the semigroup difference tends to zero
exactly when pruning creates no new stationary directions seen by the
task space \(M\): the averaged difference is governed by the stationary
subspaces of \(G|_{N}\) and \(G_{P}|_{N}\). We next connect this with
\prettyref{thm:3-1}. The relevant quantity is the smallest strictly
positive eigenvalue of $G$ on the cyclic task space \(N\).

\begin{cor}
\label{cor:ergodic-thm31}
Assume the hypotheses of \prettyref{thm:3-1}, and suppose that
\(G|_{N}\neq0\). Set
\begin{equation}
\label{eq:ergodic-thm31-1}
\mu_{N}^{+}:=\min\bigl(\sigma(G|_{N})\setminus\{0\}\bigr)>0.
\end{equation}
Let
\[
N_{0}:=\ker(G|_{N}),
\qquad
N_{+}:=N\ominus N_{0}.
\]
Then \(N_{0}\subseteq\ker(G_{P}|_{N})\), \(N_{0}\) is invariant under \(G_{P}|_{N}\),
and therefore, \(N_{+}\) is also invariant under \(G_{P}|_{N}\).
Moreover,
\begin{enumerate}[(1)]
\item
If \(R=0\) and \(\delta<1\), then
\begin{equation}
\label{eq:ergodic-thm31-2}
\ker(G_{P}|_{N})=\ker(G|_{N}),
\end{equation}
and for every \(L>0\),
\begin{equation}
\label{eq:ergodic-thm31-3}
\|\Delta_{L}\|
\le
\frac{1}{L\mu_{N}^{+}}
\left(1+\frac{1}{1-\delta}\right).
\end{equation}
\item
More generally, with \(\rho=\|R\|\) as in \prettyref{eq:3-6}, if
\begin{equation}
\label{eq:ergodic-thm31-4}
\rho<(1-\delta)\mu_{N}^{+},
\end{equation}
then
\begin{equation}
\label{eq:ergodic-thm31-5}
\ker(G_{P}|_{N})=\ker(G|_{N}),
\end{equation}
and for every \(L>0\),
\begin{equation}
\label{eq:ergodic-thm31-6}
\|\Delta_{L}\|
\le
\frac{1}{L\mu_{N}^{+}}
+
\frac{1}{L\bigl((1-\delta)\mu_{N}^{+}-\rho\bigr)}.
\end{equation}
\end{enumerate}
In either case,
\[
\Delta_{L}\to0
\qquad\text{in operator norm on }M.
\]
If \(G|_{N}=0\) instead, then \(S(t)|_{N}=S_{P}(t)|_{N}=I_{N}\), for all \(t\ge0\),
so \(\Delta_{L}=0\) for every \(L>0\).
\end{cor}

\begin{proof}
The inclusion \(N_{0}\subseteq\ker(G_{P}|_{N})\) follows from
\prettyref{eq:ergodic-raw-3}. Hence \(G_{P}|_{N}\) vanishes on \(N_{0}\), so by
self-adjointness \(N_{0}\) reduces \(G_{P}|_{N}\), and therefore \(N_{+}\) is
invariant under \(G_{P}|_{N}\).
Now let \(u\in N_{+}\). Since \(u\perp N_{0}\), we have
\[
\langle u,Gu\rangle_{K}\ge \mu_{N}^{+}\|u\|_{K}^{2}.
\]
Also, by \prettyref{eq:3-5},
\[
\langle u,Du\rangle_{K}
\ge
\langle u,(1-\delta)Gu\rangle_{K}
\ge
(1-\delta)\mu_{N}^{+}\|u\|_{K}^{2}.
\]
If \(R=0\), then \(G_{P}|_{N}=D\), so
\[
G_{P}|_{N_{+}}\ge (1-\delta)\mu_{N}^{+}I_{N_{+}}.
\]
This proves \prettyref{eq:ergodic-thm31-2}. Moreover,
\[
\left\|\frac{1}{L}\int_{0}^{L}e^{-tG|_{N_{+}}}\,dt\right\|
\le
\frac{1}{L\mu_{N}^{+}},
\qquad
\left\|\frac{1}{L}\int_{0}^{L}e^{-tG_{P}|_{N_{+}}}\,dt\right\|
\le
\frac{1}{L(1-\delta)\mu_{N}^{+}},
\]
which yields \prettyref{eq:ergodic-thm31-3} after adding the two bounds.
For the general case, write
\(B_{P}^{+}:=G_{P}|_{N_{+}}=D|_{N_{+}}+R|_{N_{+}}\). Since
\(B_{P}^{+}\) is self-adjoint and \(\|R|_{N_{+}}\|\le\rho\), we get
\[
\langle u,B_{P}^{+}u\rangle_{K}
\ge
\langle u,Du\rangle_{K}-|\langle u,Ru\rangle_{K}|
\ge
\bigl((1-\delta)\mu_{N}^{+}-\rho\bigr)\|u\|_{K}^{2}.
\]
Under \prettyref{eq:ergodic-thm31-4}, it follows that
\[
B_{P}^{+}\ge \bigl((1-\delta)\mu_{N}^{+}-\rho\bigr)I_{N_{+}},
\]
so \(B_{P}^{+}\) has no kernel on \(N_{+}\). Therefore
\prettyref{eq:ergodic-thm31-5} holds. The estimate
\prettyref{eq:ergodic-thm31-6} is proved exactly as above:
\[
\left\|\frac{1}{L}\int_{0}^{L}e^{-tB_{P}^{+}}\,dt\right\|
\le
\frac{1}{L\bigl((1-\delta)\mu_{N}^{+}-\rho\bigr)}.
\]
Combining this with the bound for \(G|_{N_{+}}\) yields
\prettyref{eq:ergodic-thm31-6}.
The final convergence statement follows from either
\prettyref{eq:ergodic-thm31-3} or \prettyref{eq:ergodic-thm31-6}. If
\(G|_{N}=0\), then \(T^{*}u=0\) for all \(u\in N\), hence
\(G_{P}|_{N}=TPT^{*}|_{N}=0\) as well, and the last claim is immediate.
\end{proof}

\begin{rem}
\label{rem:ergodic-complement}
The finite-time estimate \prettyref{eq:3-7} and the Ces\`aro limit above
measure different quantities. Averaging the right-hand side of
\prettyref{eq:3-7} over \([0,L]\) gives a quantity of order \(1+L\rho\),
which carries no asymptotic information; the Ces\`aro limit is instead
controlled by the stationary subspaces of \(G|_{N}\) and \(G_{P}|_{N}\).
We also note that the results of this section are statements about the
frozen model. In finite-width training the empirical kernel may drift,
and the limit \(L\to\infty\) should not be read as a long-time statement
about the nonlinear dynamics. The time-varying problem is taken up in
\prettyref{sec:na-ntk} by different methods; the evolution there is a
product of local semigroup factors rather than a single semigroup, and
the mean ergodic theorem for \(C_{0}\)-semigroups does not apply
directly.
\end{rem}

\section{Ces\`aro estimates without a lower spectral edge}
\label{sec:spectral-distribution-replacement}

The explicit decay estimates in Corollaries \ref{cor:ergodic-mixed} and
\ref{cor:ergodic-thm31} use a positive lower spectral edge on the
nonstationary part of the task space. We show here that the lower-edge
assumption can be replaced by weaker, task-dependent conditions on the
low-frequency spectral distribution. Under these conditions, the decay of
the nonstationary components yields a direct approximation
of the original average by the pruned average whenever their stationary
projections agree on $M$.
Throughout this section, let \(N\subseteq K\) be a closed \(G\)-invariant
subspace with \(M\subseteq N\), and set
\[
B:=G|_{N}.
\]
The canonical minimal choice for the results involving only the unpruned
dynamics is $N=\mathcal C_G(M)$. We retain a general $N$ because
the comparison with the pruned dynamics may require a larger space reducing
both $G$ and $G_P$.
Since \(G\) is self-adjoint, the closed invariant subspace \(N\) reduces \(G\),
and hence \(B\) is a bounded positive self-adjoint operator on \(N\). Let $P_{\ker B}$
denote the orthogonal projection onto the stationary subspace of the
unpruned dynamics on \(N\). For \(v\in N\), let \(\mu_{v}^{B}\) denote the
spectral measure of \(B\) associated with \(v\), so that
\[
\langle f(B)v,v\rangle
=
\int_{[0,\|B\|]} f(\lambda)\,d\mu_{v}^{B}(\lambda),
\]
for every bounded Borel function \(f\).
Since
\[
\ker B\subseteq \ker G=\ker T^{*},
\]
we have
\[
(I-P)T^{*}P_{\ker B}=0,
\]
so the zero spectral atom of \(B\) does not contribute to the tangent
leakage. The hypotheses below are therefore imposed on the nonzero
spectral mass \(\mu_{v}^{B}((0,r])\), not on the full mass
\(\mu_{v}^{B}([0,r])\). If \(B=0\), then the nonstationary quantities
below vanish and the estimates are trivial; we accordingly assume
\(\|B\|>0\).

\subsection{Averaged energy}

The quantity appearing in the mixed leakage estimate of
\prettyref{cor:ergodic-mixed}, and again in the estimates below, is the
averaged energy
\[
I_{L}^{B}(v):=
\frac1L\int_{0}^{L}\|e^{-tB}v\|_{K}^{2}\,dt,
\qquad v\in N,
\]
with nonstationary part
\[
I_{L}^{B,+}(v):=
\frac1L\int_{0}^{L}\|e^{-tB}(I-P_{\ker B})v\|_{K}^{2}\,dt.
\]
Since \(N\) is \(G\)-invariant, for \(v\in N\) this is the same as computing the
unpruned energy of \(S(t)v=e^{-tG}v\) inside the reduced space \(N\).

\begin{thm}
\label{thm:spectral-distribution-energy}
For every \(v\in N\) and every \(L>0\),
\begin{equation}
\label{eq:spectral-distribution-energy-1}
I_{L}^{B}(v)
=
\langle \phi_{L}(B)v,v\rangle_{K}
=
\int_{[0,\|B\|]}\phi_{L}(\lambda)\,d\mu_{v}^{B}(\lambda),
\end{equation}
where
\begin{equation}
\label{eq:spectral-distribution-energy-2}
\phi_{L}(\lambda):=
\begin{cases}
\dfrac{1-e^{-2L\lambda}}{2L\lambda}, & \lambda>0,\\[1ex]
1, & \lambda=0.
\end{cases}
\end{equation}
Moreover,
\begin{equation}
\label{eq:spectral-distribution-energy-3}
I_{L}^{B,+}(v)
=
\int_{(0,\|B\|]}\phi_{L}(\lambda)\,d\mu_{v}^{B}(\lambda)
=
I_{L}^{B}(v)-\|P_{\ker B}v\|_{K}^{2}.
\end{equation}
\end{thm}

\begin{proof}
Since \(B\) is self-adjoint and positive,
\[
\|e^{-tB}v\|_{K}^{2}
=
\langle e^{-2tB}v,v\rangle_{K}, 
\]
so that 
\[
I_{L}^{B}(v)
=
\left\langle
\left(
\frac1L\int_{0}^{L}e^{-2tB}\,dt
\right)v,
v
\right\rangle_{K}.
\]
Then, for each scalar \(\lambda\ge0\),
\[
\frac1L\int_{0}^{L}e^{-2t\lambda}\,dt
=
\begin{cases}
\dfrac{1-e^{-2L\lambda}}{2L\lambda}, & \lambda>0,\\[1ex]
1, & \lambda=0,
\end{cases}
=
\phi_{L}(\lambda).
\]
Therefore
\[
\frac1L\int_{0}^{L}e^{-2tB}\,dt=\phi_{L}(B),
\]
which proves the first equality in \prettyref{eq:spectral-distribution-energy-1}.
The spectral theorem gives the integral representation in
\prettyref{eq:spectral-distribution-energy-1}. Since
\(e^{-tB}P_{\ker B}=P_{\ker B}\) and \((I-P_{\ker B})\) removes exactly the atom at
\(0\), \prettyref{eq:spectral-distribution-energy-3} follows.
\end{proof}

\begin{rem}
The averaged energy also admits a contour-resolvent representation.
The function \(\phi_{L}\) extends to an entire function, so if
\(\Gamma\) is a positively oriented contour in the resolvent set of
\(B\) enclosing \(\sigma(B)\), the holomorphic functional calculus gives
\[
\phi_{L}(B)
=
\frac{1}{2\pi i}\int_{\Gamma}\phi_{L}(z)(zI-B)^{-1}\,dz,
\qquad\text{and hence,}\qquad
I_{L}^{B}(v)
=
\frac{1}{2\pi i}\int_{\Gamma}\phi_{L}(z)\,m_{v}^{B}(z)\,dz,
\]
where \(m_{v}^{B}(z):=\langle (zI-B)^{-1}v,v\rangle_{K}\) is the
Stieltjes transform of \(\mu_{v}^{B}\). Only the spectral-measure form
\prettyref{eq:spectral-distribution-energy-1} is used below.
\end{rem}

For later use, note that for every \(\lambda>0\),
\begin{equation}
\label{eq:spectral-distribution-energy-8}
0\le \phi_{L}(\lambda)\le \min\left\{1,\frac{1}{2L\lambda}\right\}.
\end{equation}

\subsection{Low-frequency spectral mass}

Define the nonzero low-frequency spectral mass of \(v\) by
\begin{equation}
\label{eq:spectral-distribution-nonzero-mass}
\widetilde F_{v}^{B}(r):=\mu_{v}^{B}((0,r]),
\qquad 0<r\le \|B\|.
\end{equation}
This quantity measures how much of \(v\) lies in the nonzero spectral window
\((0,r]\).

\begin{thm}
\label{thm:spectral-distribution-rates}
Assume that \(B\ne0\), and fix \(v\in N\). Suppose that there exist
constants \(C_{v}>0\) and \(\gamma>0\) such that
\begin{equation}
\label{eq:spectral-distribution-rates-1}
\widetilde F_{v}^{B}(r)
\le C_{v}r^{\gamma},
\qquad 0<r\le \|B\|.
\end{equation}
Then there exists a constant \(C_{\gamma,\|B\|}>0\), depending only on
\(\gamma\) and \(\|B\|\), such that for all \(L\ge \|B\|^{-1}\),
\begin{equation}
\label{eq:spectral-distribution-rates-2}
I_{L}^{B,+}(v) = \int_{(0,\|B\|]}\phi_{L}(\lambda)\,d\mu_{v}^{B}(\lambda)
\le
C_{\gamma,\|B\|}C_{v}\times
\begin{cases}
L^{-\gamma}, & 0<\gamma<1,\\[1ex]
L^{-1}\log(1+L\|B\|), & \gamma=1,\\[1ex]
L^{-1}, & \gamma>1.
\end{cases}
\end{equation}
\end{thm}

\begin{proof}
By \prettyref{thm:spectral-distribution-energy},
\[
I_{L}^{B,+}(v)
=
\int_{(0,\|B\|]}\phi_{L}(\lambda)\,d\mu_{v}^{B}(\lambda),
\]
and by splitting the integral at
the edge scale \(L^{-1}\), we have
\[
I_{L}^{B,+}(v)
=
\int_{(0,L^{-1}]}\phi_{L}(\lambda)\,d\mu_{v}^{B}(\lambda)
+
\int_{(L^{-1},\|B\|]}\phi_{L}(\lambda)\,d\mu_{v}^{B}(\lambda)
=:
I_{1}+I_{2}.
\]
For \(I_{1}\), use \(\phi_{L}(\lambda)\le1\):
\[
I_{1}
\le
\mu_{v}^{B}((0,L^{-1}])
=
\widetilde F_{v}^{B}(L^{-1})
\le
C_{v}L^{-\gamma}.
\]
For \(I_{2}\), use \prettyref{eq:spectral-distribution-energy-8}:
\[
I_{2}
\le
\frac{1}{2L}
\int_{(L^{-1},\|B\|]}\lambda^{-1}\,d\mu_{v}^{B}(\lambda)
=
\frac{1}{2L}
\int_{L^{-1}}^{\|B\|}\lambda^{-1}\,d\widetilde F_{v}^{B}(\lambda).
\]
A Stieltjes integration by parts gives, for \(0<a<b\le\|B\|\),
\[
\int_{a}^{b}\lambda^{-1}\,d\widetilde F_{v}^{B}(\lambda)
=
\Bigl[\lambda^{-1}\widetilde F_{v}^{B}(\lambda)\Bigr]_{a}^{b}
+
\int_{a}^{b}\lambda^{-2}\widetilde F_{v}^{B}(\lambda)\,d\lambda.
\]
Taking \(a=L^{-1}\) and \(b=\|B\|\), dropping the nonpositive lower-boundary
term, and using \prettyref{eq:spectral-distribution-rates-1}, we get
\[
\int_{L^{-1}}^{\|B\|}\lambda^{-1}\,d\widetilde F_{v}^{B}(\lambda)
\le
C_{v}\|B\|^{\gamma-1}
+
C_{v}\int_{L^{-1}}^{\|B\|}\lambda^{\gamma-2}\,d\lambda.
\]
Thus,
\begin{equation}
\label{eq:spectral-distribution-rates-proof-1}
I_{2}
\le
\frac{C_{v}}{2L}
\left(
\|B\|^{\gamma-1}
+
\int_{L^{-1}}^{\|B\|}\lambda^{\gamma-2}\,d\lambda
\right).
\end{equation}
If \(0<\gamma<1\), then
\[
\int_{L^{-1}}^{\|B\|}\lambda^{\gamma-2}\,d\lambda
=
\frac{\|B\|^{\gamma-1}-L^{1-\gamma}}{\gamma-1}
\le
\frac{1}{1-\gamma}L^{1-\gamma},
\]
and \(L\ge\|B\|^{-1}\) also implies
\(\|B\|^{\gamma-1}\le L^{1-\gamma}\).  Hence,
\(I_{2}\le C'_{\gamma,\|B\|}C_{v}L^{-\gamma}\).
If \(\gamma=1\), then
\[
\int_{L^{-1}}^{\|B\|}\lambda^{-1}\,d\lambda
=
\log(L\|B\|),
\]
so
\[
I_{2}\le C'_{\|B\|}C_{v}L^{-1}\log(1+L\|B\|).
\]
If \(\gamma>1\), then
\[
\int_{L^{-1}}^{\|B\|}\lambda^{\gamma-2}\,d\lambda
=
\frac{\|B\|^{\gamma-1}-L^{1-\gamma}}{\gamma-1}
\le
\frac{\|B\|^{\gamma-1}}{\gamma-1},
\]
and \(L^{-\gamma}\le \|B\|^{\gamma-1}L^{-1}\).  Hence
\(I_{1}+I_{2}\le C'_{\gamma,\|B\|}C_{v}L^{-1}\).
Combining the estimates for \(I_{1}\) and \(I_{2}\) proves
\prettyref{eq:spectral-distribution-rates-2}.
\end{proof}

\subsection{Averaged leakage without a lower edge}

We now combine \prettyref{thm:spectral-distribution-rates} with the
mixed task-tail condition from \prettyref{sec:ergodic-task}.

\begin{cor}
\label{cor:spectral-distribution-ergodic}
Assume the mixed task-tail condition \prettyref{eq:ergodic-mixed-def} on \(N\) is given by
\[
\|(I-P)T^{*}u\|_{H}^{2}
\le
\alpha\|T^{*}u\|_{H}^{2}+\beta\|u\|_{K}^{2},
\qquad u\in N.
\]
Suppose further that there exist constants \(C_{M}>0\) and \(\gamma>0\) such that
for every \(v\in M\),
\begin{equation}
\label{eq:spectral-distribution-ergodic-1}
\widetilde F_{v}^{B}(r)
=
\mu_{v}^{B}((0,r])
\le C_{M}r^{\gamma}\|v\|_{K}^{2},
\qquad 0<r\le \|B\|.
\end{equation}
Then there exists \(C_{\gamma,\|B\|}>0\) such that for all
\(L\ge \|B\|^{-1}\),
\begin{equation}
\label{eq:spectral-distribution-ergodic-2}
\|(I-P)T^{*}C_{L}|_{M}\|^{2}
\le
\frac{\alpha}{2L}
+
\beta C_{M}C_{\gamma,\|B\|}\times
\begin{cases}
L^{-\gamma}, & 0<\gamma<1,\\[1ex]
L^{-1}\log(1+L\|B\|), & \gamma=1,\\[1ex]
L^{-1}, & \gamma>1.
\end{cases}
\end{equation}
\end{cor}

\begin{proof}
Since \(\ker B\subseteq\ker T^{*}\), we have
\[
(I-P)T^{*}P_{\ker B}=0.
\]
Also \(C_{L}P_{\ker B}=P_{\ker B}\), and therefore,
\[
(I-P)T^{*}C_{L}v
=
(I-P)T^{*}C_{L}(I-P_{\ker B})v.
\]
By Jensen's inequality,
\[
\|(I-P)T^{*}C_{L}v\|_{H}^{2}
\le
\frac1L\int_{0}^{L}
\|(I-P)T^{*}S(t)(I-P_{\ker B})v\|_{H}^{2}\,dt.
\]
Applying the mixed task-tail condition to
\(u=S(t)(I-P_{\ker B})v\in N\), we obtain
\begin{align*}
\|(I-P)T^{*}C_{L}v\|_{H}^{2}
&\le
\frac{\alpha}{L}\int_{0}^{L}
\|T^{*}S(t)(I-P_{\ker B})v\|_{H}^{2}\,dt \\
&\quad+
\frac{\beta}{L}\int_{0}^{L}
\|S(t)(I-P_{\ker B})v\|_{K}^{2}\,dt.
\end{align*}
For the first integral, since \(\|T^{*}u\|_{H}^{2}=\langle
u,Bu\rangle_{K}=\|B^{1/2}u\|_{K}^{2}\) for \(u\in N\), the energy
identity gives
\[
\begin{aligned}
\int_{0}^{L}\|T^{*}S(t)(I-P_{\ker B})v\|_{H}^{2}\,dt
&=
\int_{0}^{L}\|B^{1/2}e^{-tB}(I-P_{\ker B})v\|_{K}^{2}\,dt \\
&=
\frac12\|(I-P_{\ker B})v\|_{K}^{2}
-\frac12\|e^{-LB}(I-P_{\ker B})v\|_{K}^{2}
\le
\frac12\|v\|_{K}^{2}.
\end{aligned}
\]
The second integral is \(L I_{L}^{B,+}(v)\).  By
\prettyref{thm:spectral-distribution-rates} and
\prettyref{eq:spectral-distribution-ergodic-1},
\[
I_{L}^{B,+}(v)
\le
C_{M}C_{\gamma,\|B\|}
\begin{cases}
L^{-\gamma}, & 0<\gamma<1,\\[1ex]
L^{-1}\log(1+L\|B\|), & \gamma=1,\\[1ex]
L^{-1}, & \gamma>1
\end{cases}
\|v\|_{K}^{2}.
\]
Taking the supremum over unit vectors \(v\in M\) gives
\prettyref{eq:spectral-distribution-ergodic-2}.
\end{proof}

Thus the corollary controls the component of the Ces\`aro-averaged
unpruned tangent vector \(T^{*}C_{L}v\) lying in the discarded parameter
directions. Under taskwise control of the nonzero spectral mass near the
origin, this tangent vector becomes asymptotically concentrated in the
retained parameter space \(\operatorname{ran}P\).

\subsection{A Ces\`aro approximation theorem}

We conclude the section with a quantitative comparison of the Ces\`aro
averaged semigroups. The general estimate controls the difference of
the pruned and original averages after subtracting the exact stationary
component of their difference; when the stationary projections agree on $M$, it gives a 
direct operator-norm approximation with an explicit rate.

\begin{thm}
\label{thm:spectral-distribution-cesaro-approx}
Let \(N\subseteq K\) be a closed subspace with \(M\subseteq N\).  Assume that
\(N\) reduces both \(G\) and \(G_{P}\), and set
\[
B:=G|_{N},
\qquad
B_{P}:=G_{P}|_{N}.
\]
Let
\[
C_{L}:=\frac1L\int_{0}^{L}e^{-tG}\,dt,
\qquad
C_{P,L}:=\frac1L\int_{0}^{L}e^{-tG_{P}}\,dt.
\]
For \(v\in M\), let \(\mu_{v}^{B}\) and \(\mu_{v}^{B_{P}}\) denote the spectral
measures of \(B\) and \(B_{P}\), respectively.  Suppose that there are constants
\(C_{M},C_{M,P}>0\) and exponents \(\gamma,\gamma_{P}>0\) such that, for every
\(v\in M\),
\begin{equation}
\label{eq:spectral-distribution-cesaro-approx-1}
\mu_{v}^{B}((0,r])
\le
C_{M}r^{\gamma}\|v\|_{K}^{2},
\qquad
0<r\le \|B\|,
\end{equation}
and
\begin{equation}
\label{eq:spectral-distribution-cesaro-approx-2}
\mu_{v}^{B_{P}}((0,r])
\le
C_{M,P}r^{\gamma_{P}}\|v\|_{K}^{2},
\qquad
0<r\le \|B_{P}\|.
\end{equation}
For \(a>0\), define
\[
\mathcal R_{\omega,a}(L)
:=
\begin{cases}
L^{-\omega}, & 0<\omega<1,\\[1ex]
L^{-1}\log(1+La), & \omega=1,\\[1ex]
L^{-1}, & \omega>1.
\end{cases}
\]
Let
\[
L_{0}:=\max\{a^{-1}:a\in\{\|B\|,\|B_{P}\|\},\ a>0\},
\]
with the convention that \(L_{0}=0\) if both \(B\) and \(B_{P}\) are zero.  Then,
for all \(L\ge L_{0}\),
\begin{align}
\label{eq:spectral-distribution-cesaro-approx-3}
&
\left\|
(C_{P,L}-C_{L})|_{M}
-
\bigl(P_{\ker B_{P}}-P_{\ker B}\bigr)|_{M}
\right\|
\nonumber\\
&\qquad\le
\left(C_{M,P}C_{\gamma_{P},\|B_{P}\|}
\mathcal R_{\gamma_{P},\|B_{P}\|}(L)\right)^{1/2}
+
\left(C_{M}C_{\gamma,\|B\|}
\mathcal R_{\gamma,\|B\|}(L)\right)^{1/2},
\end{align}
where the corresponding term is understood to be zero if \(B=0\) or
\(B_{P}=0\).
In particular, if
\begin{equation}
\label{eq:spectral-distribution-cesaro-approx-4}
P_{\ker B_{P}}|_{M}=P_{\ker B}|_{M},
\end{equation}
then
\begin{align}
\label{eq:spectral-distribution-cesaro-approx-5}
\|(C_{P,L}-C_{L})|_{M}\|
&\le
\left(C_{M,P}C_{\gamma_{P},\|B_{P}\|}
\mathcal R_{\gamma_{P},\|B_{P}\|}(L)\right)^{1/2}
\nonumber\\
&\quad+
\left(C_{M}C_{\gamma,\|B\|}
\mathcal R_{\gamma,\|B\|}(L)\right)^{1/2},
\end{align}
again omitting any zero-operator term.  Consequently,
\[
(C_{P,L}-C_{L})|_{M}\to0,
\qquad
\text{in operator norm as }L\to\infty.
\]
\end{thm}

\begin{proof}
Recall that for \(L>0\),
\[
\psi_{L}(\lambda)
:=
\begin{cases}
\dfrac{1-e^{-L\lambda}}{L\lambda}, & \lambda>0,\\[1ex]
1, & \lambda=0.
\end{cases}
\]
Hence, by the functional calculus for bounded positive
self-adjoint operators,
\[
C_{L}|_{N}
=
\frac1L\int_{0}^{L}e^{-tB}\,dt
=
\psi_{L}(B),
\qquad
C_{P,L}|_{N}
=
\frac1L\int_{0}^{L}e^{-tB_{P}}\,dt
=
\psi_{L}(B_{P}).
\]
We first obtain a convergence estimate for a single operator.
Let \(X\) denote either \(B\) or \(B_{P}\), and let \(E^{X}\) be its
spectral resolution. Since \(X\ge0\), its spectrum is contained in
\([0,\|X\|]\), and
\[
P_{\ker X}=E^{X}(\{0\}).
\]
Because \(\psi_{L}(0)=1\), subtracting the stationary projection removes
exactly the zero spectral atom. Thus, for every \(v\in N\),
\[
\bigl(\psi_{L}(X)-P_{\ker X}\bigr)v
=
\int_{(0,\|X\|]}\psi_{L}(\lambda)\,dE^{X}(\lambda)v.
\]
Consequently, if \(\mu_{v}^{X}\) denotes the spectral measure
\[
\mu_{v}^{X}(\Omega):=\|E^{X}(\Omega)v\|^{2},
\]
then the spectral theorem gives
\[
\|\bigl(\psi_{L}(X)-P_{\ker X}\bigr)v\|_{K}^{2}
=
\int_{(0,\|X\|]}\psi_{L}(\lambda)^{2}\,d\mu_{v}^{X}(\lambda).
\]
We now compare \(\psi_{L}\) with the
multiplier used in \prettyref{thm:spectral-distribution-rates}. For
\(\lambda>0\),
\[
\psi_{L}(\lambda)
=
\frac1L\int_{0}^{L}e^{-t\lambda}\,dt.
\]
By Jensen's inequality,
\[
\psi_{L}(\lambda)^{2}
=
\left(\frac1L\int_{0}^{L}e^{-t\lambda}\,dt\right)^{2}
\le
\frac1L\int_{0}^{L}e^{-2t\lambda}\,dt.
\]
Therefore,
\[
0\le \psi_{L}(\lambda)^{2}
\le
\phi_{L}(\lambda),
\qquad \lambda>0,
\]
with \(\phi_{L}\) as in \prettyref{eq:spectral-distribution-energy-2}.
Hence, for every \(v\in M\),
\[
\|\bigl(\psi_{L}(X)-P_{\ker X}\bigr)v\|_{K}^{2}
\le
\int_{(0,\|X\|]}\phi_{L}(\lambda)\,d\mu_{v}^{X}(\lambda).
\]
Applying the spectral-distribution estimate from
\prettyref{thm:spectral-distribution-rates} first with \(X=B\), and
then with \(X=B_{P}\), gives
\[
\|\bigl(\psi_{L}(B)-P_{\ker B}\bigr)|_{M}\|^{2}
\le
C_{M}C_{\gamma,\|B\|}
\mathcal R_{\gamma,\|B\|}(L),
\]
and
\[
\|\bigl(\psi_{L}(B_{P})-P_{\ker B_{P}}\bigr)|_{M}\|^{2}
\le
C_{M,P}C_{\gamma_{P},\|B_{P}\|}
\mathcal R_{\gamma_{P},\|B_{P}\|}(L).
\]
It remains to compare the two Ces\`aro averages. On \(N\), we have
\[
C_{L}=\psi_{L}(B),
\qquad
C_{P,L}=\psi_{L}(B_{P}).
\]
Therefore, restricting to \(M\) and using the triangle inequality gives
\[
\begin{aligned}
&\left\|
\left(
C_{P,L}-C_{L}
-
(P_{\ker B_{P}}-P_{\ker B})
\right)\big|_{M}
\right\|  \\
&\qquad \le
\|\bigl(\psi_{L}(B_{P})-P_{\ker B_{P}}\bigr)|_{M}\|
+
\|\bigl(\psi_{L}(B)-P_{\ker B}\bigr)|_{M}\|  \\
&\qquad \le
\left(
C_{M,P}C_{\gamma_{P},\|B_{P}\|}
\mathcal R_{\gamma_{P},\|B_{P}\|}(L)
\right)^{1/2}
+
\left(
C_{M}C_{\gamma,\|B\|}
\mathcal R_{\gamma,\|B\|}(L)
\right)^{1/2}.
\end{aligned}
\]
This proves \prettyref{eq:spectral-distribution-cesaro-approx-3}.
Finally, suppose that \prettyref{eq:spectral-distribution-cesaro-approx-4}
holds, that is,
\[
(P_{\ker B_{P}}-P_{\ker B})|_{M}=0.
\]
Then the stationary projection term in the previous estimate vanishes on
\(M\), so
\[
\|(C_{P,L}-C_{L})|_{M}\|
=
\left\|
\left(
C_{P,L}-C_{L}
-
(P_{\ker B_{P}}-P_{\ker B})
\right)\big|_{M}
\right\|,
\]
and the preceding bound yields
\prettyref{eq:spectral-distribution-cesaro-approx-5}.
\end{proof}

Condition \prettyref{eq:spectral-distribution-cesaro-approx-4} states
that pruning creates no new stationary directions visible from \(M\).
When it holds, and the taskwise spectral mass near zero satisfies the
power bounds above, the averaged pruned dynamics \(C_{P,L}|_{M}\)
converges in operator norm to the averaged original dynamics
\(C_{L}|_{M}\), with a rate determined by the exponents at the origin.

\section{A finite-time estimate for unbounded generators}
\label{sec:unbounded-kato}

In the preceding sections, we considered the bounded generator
\[
G=TT^{*}\in\mathcal B(K),
\]
and its pruned counterpart
\[
G_{P}=TPT^{*}.
\]
In that setting, every operator is defined on all of the output Hilbert space
\(K\), the difference \(G-G_{P}\) is bounded, and, under the task-capture
hypothesis \prettyref{eq:2-2} on \(N=\mathcal C_G(M)\), Duhamel's formula may
be combined directly with the estimate
\begin{equation}
\label{eq:unbd-bounded-key-estimate}
\|(G-G_{P})u\|
\le
\epsilon\sqrt{\|G-G_{P}\|}\,\|G^{1/2}u\|,
\qquad u\in N.
\end{equation}
The purpose of this section is to show that the Duhamel--energy argument
developed in \prettyref{sec:3} extends to output-side generators that
need not be bounded, with essentially the same argument once a suitable
variation-of-constants formula and an orbitwise relative square-root
estimate for the perturbation are available. We restrict attention to
this finite-time extension.


Unbounded generators arise naturally from derivative-weighted losses.
Let
\[
\mathcal D:D(\mathcal D)\subset K_{0}\to Y
\]
be a densely defined closed operator from a base output space \(K_{0}\)
into a feature space \(Y\).
Let
\[
T:H\to K_{0}
\]
be the frozen Jacobian, and set \(G=TT^{*}\in\mathcal B(K_{0})\). Define
the derivative-feature Jacobian
\begin{equation}
\label{eq:unbd-JL-def}
J_{\mathcal D}:=\mathcal DT,
\qquad
D(J_{\mathcal D}):=\{h\in H:Th\in D(\mathcal D)\},
\end{equation}

The closedness of \(\mathcal D\) and the boundedness of \(T\) imply that
\(J_{\mathcal D}\) is closed; we assume that its domain is dense in \(H\).
If \(F(\theta_{0})-y\in D(\mathcal D)\),
the \(\mathcal D\)-weighted loss is
\[
\mathcal L_{\mathcal D}(h)
=
\frac12\|J_{\mathcal D}h+\mathcal D(F(\theta_{0})-y)\|_{Y}^{2},
\qquad h\in D(J_{\mathcal D}),
\]
with derivative residual
\(r_{\mathcal D}=J_{\mathcal D}h+\mathcal D(F(\theta_{0})-y)\). The
formal parameter gradient flow \(h'(t)=-J_{\mathcal D}^{*}r_{\mathcal D}(t)\)
leads to the residual equation
\begin{equation}
\label{eq:unbd-derivative-residual-flow}
r_{\mathcal D}'(t)=-J_{\mathcal D}J_{\mathcal D}^{*}r_{\mathcal D}(t).
\end{equation}
The task-space generator associated with the \(\mathcal D\)-weighted
residual is therefore
\begin{equation}
\label{eq:unbd-AL-rigorous}
A_{\mathcal D}:=J_{\mathcal D}J_{\mathcal D}^{*},
\end{equation}
a nonnegative self-adjoint operator on the derivative-feature space
\(Y\). Since \(T\) is bounded, one always has
\(T^{*}\mathcal D^{*}\subseteq J_{\mathcal D}^{*}\), and on any common
domain where \(J_{\mathcal D}^{*}r=T^{*}\mathcal D^{*}r\) and
\(TT^{*}\mathcal D^{*}r\in D(\mathcal D)\), the generator acts as
\[
A_{\mathcal D}r=\mathcal DTT^{*}\mathcal D^{*}r=\mathcal DG\mathcal D^{*}r,
\]
the formal expression \(\mathcal DG\mathcal D^{*}\) being understood on
the natural domain
\begin{equation}
\label{eq:unbd-LGL-domain}
D(\mathcal DG\mathcal D^{*})
=
\{r\in D(\mathcal D^{*}):G\mathcal D^{*}r\in D(\mathcal D)\}.
\end{equation}
We take \(A_{\mathcal D}=J_{\mathcal D}J_{\mathcal D}^{*}\) as the
definition; the expression \(\mathcal DG\mathcal D^{*}\) is a concrete
formula for it when the indicated domains and adjoint identities are
valid.

\begin{rem}
For finite-dimensional images, first- and second-order difference
operators are finite matrices, and the corresponding derivative-task
generator is bounded; the unbounded model is the continuum, or
high-resolution, limit. As the pixel spacing \(h\downarrow0\), the norms
of first-difference and second-difference operators typically scale in the order of 
\(h^{-1}\) and \(h^{-2}\), respectively. Hence, the generators of the form
\(\mathcal D_{h}G_{h}\mathcal D_{h}^{*}\) may have norms growing with
resolution unless the kernel \(G_{h}\) provides sufficient
high-frequency smoothing.
\end{rem}

For the rest of the section, \(K\) is a Hilbert space and
\(A\ge0\) is a self-adjoint, possibly unbounded, operator on \(K\).  We write
\[
S(t)=e^{-tA},\qquad t\ge0,
\]
for the associated contraction semigroup.  The perturbed generator is denoted by
\(A_{\#}\ge0\), with semigroup
\[
S_{\#}(t)=e^{-tA_{\#}},\qquad t\ge0;
\]
it may model pruned, filtered, or otherwise perturbed unbounded task
generators. We do not develop a systematic perturbation theory for such
generators here. In particular, the construction of \(A_{\#}\) from a
pruning mechanism, verification of the required domain or form-domain
conditions and of the variation-of-constants formula, and unbounded
analogues of the averaged and nonautonomous results are left for future
work. Such a development would require a separate analysis beyond the
core bounded task-space theory developed in this paper.

The theorem below assumes that \(A_{\#}\) is sel-fadjoint and that the
stated variation-of-constants formula is valid on the task orbit. Only
the vectors reached from \(M\) under the unperturbed semigroup enter the
argument, so the relative bound is imposed directly along those
trajectories.

\begin{thm}
\label{thm:kato-relative-finite-time}
Let \(A\ge0\) and \(A_{\#}\ge0\) be self-adjoint operators on \(K\), with
semigroups \(S(t)\) and \(S_{\#}(t)\). Let \(M\subseteq K\) be closed
and let \(\tau>0\). Let \(W:D(W)\subseteq K\to K\) be a linear operator
such that
\begin{equation}
\label{eq:kato-orbit-domain}
S(s)M\subseteq D(W),
\qquad 0<s\le\tau.
\end{equation}
The sign convention below corresponds to \(W=A-A_{\#}\) whenever this
operator difference and the usual variation-of-constants formula are
well defined. Assume that
\begin{equation}
\label{eq:kato-duhamel}
S_{\#}(t)v-S(t)v
=
\int_{0}^{t}S_{\#}(t-s)W S(s)v\,ds
\end{equation}
holds as a Bochner integral for \(v\in M\) and \(0<t\le\tau\).
Suppose also that there exists \(\eta\ge0\) such that
\begin{equation}
\label{eq:kato-half-bound}
\|W S(s)v\|_{K}
\le
\eta\|A^{1/2}S(s)v\|_{K},
\qquad
v\in M,\quad 0<s\le\tau.
\end{equation}
Then
\begin{equation}
\label{eq:kato-finite-bound}
\sup_{0\le t\le\tau}
\|(S_{\#}(t)-S(t))|_{M}\|
\le
\eta\sqrt{\frac{\tau}{2}}.
\end{equation}
If, in addition, \(M\subseteq D(A^{1/2})\), then the restriction
\(A^{1/2}|_{M}:M\to K\) is bounded. Setting
\begin{equation}
\label{eq:kato-task-energy}
\kappa_{M}:=\|A^{1/2}|_{M}\|,
\end{equation}
one has
\begin{equation}
\label{eq:kato-finite-bound-refined}
\sup_{0\le t\le\tau}
\|(S_{\#}(t)-S(t))|_{M}\|
\le
\eta
\min\left\{
\tau\kappa_{M},
\sqrt{\frac{\tau}{2}}
\right\}.
\end{equation}
\end{thm}

\begin{proof}
Since \(A_{\#}\ge0\), the semigroup \(S_{\#}(t)\) is contractive. By
\prettyref{eq:kato-duhamel}, for \(v\in M\),
\[
\|(S_{\#}(t)-S(t))v\|
\le
\int_{0}^{t}\|S_{\#}(t-s)W S(s)v\|\,ds.
\]
For \(s>0\), the vector \(S(s)v\) belongs to \(D(A^{1/2})\). 
Indeed,
by the spectral theorem,
\(A^{1/2}e^{-sA}\) is the bounded spectral multiplier
\(\lambda\mapsto \lambda^{1/2}e^{-s\lambda}\), and
\[
\|A^{1/2}e^{-sA}\|
=
\sup_{\lambda\ge0}\lambda^{1/2}e^{-s\lambda}
\le
(2es)^{-1/2}<\infty.
\]
Thus \prettyref{eq:kato-half-bound} applies for every \(s>0\). The value
at \(s=0\) is irrelevant for the integral. Using again the contractivity
of \(S_{\#}(t-s)\), we obtain
\[
\|(S_{\#}(t)-S(t))v\|
\le
\eta\int_{0}^{t}\|A^{1/2}S(s)v\|\,ds.
\]
By Cauchy--Schwarz,
\[
\int_{0}^{t}\|A^{1/2}S(s)v\|\,ds
\le
\sqrt t
\left(
\int_{0}^{t}\|A^{1/2}S(s)v\|^{2}\,ds
\right)^{1/2},
\]
and again by the spectral theorem,
\[
\int_{0}^{t}\|A^{1/2}S(s)v\|^{2}\,ds
=
\frac12
\left(
\|v\|^{2}-\|S(t)v\|^{2}
\right)
\le
\frac12\|v\|^{2}.
\]
Hence,
\[
\|(S_{\#}(t)-S(t))v\|
\le
\eta\sqrt{\frac{t}{2}}\|v\|.
\]
Taking the supremum over \(0\le t\le\tau\) and over unit vectors \(v\in M\)
gives \prettyref{eq:kato-finite-bound}.

Suppose now that \(M\subseteq D(A^{1/2})\). Since \(M\) is closed and
\(A^{1/2}\) is closed, the closed graph theorem shows that
\(\kappa_M<\infty\). The corresponding functional calculus gives
\[
A^{1/2}S(s)v=S(s)A^{1/2}v,
\qquad v\in M,\quad s\ge0,
\]
and the contractivity of \(S(s)\) gives 
\[
\|A^{1/2}S(s)v\|
\le
\|A^{1/2}v\|
\le
\kappa_M\|v\|.
\]
Consequently,
\[
\int_{0}^{t}\|A^{1/2}S(s)v\|\,ds
\le
t\kappa_M\|v\|.
\]
Combining this estimate with the energy estimate already proved yields
\prettyref{eq:kato-finite-bound-refined}.
\end{proof}

\begin{rem}
Let \(E_A\) denote the spectral measure of \(A\). If
\(M\subseteq E_A([0,\Lambda])K\) for some \(\Lambda<\infty\), then
\(M\subseteq D(A^{1/2})\) and \(\kappa_M\le\sqrt{\Lambda}\). Thus
\prettyref{eq:kato-finite-bound-refined} also gives the same estimate
with \(\tau\sqrt{\Lambda}\) in place of \(\tau\kappa_M\).

For bounded \(A=G\), \(A_{\#}=G_{P}\), and \(W=G-G_{P}\), the
capture hypothesis \prettyref{eq:2-2} implies
\prettyref{eq:kato-half-bound} with
\(\eta=\epsilon\sqrt{\|G-G_{P}\|}\), because
\(S(s)M\subseteq\mathcal C_G(M)\). Moreover, if
\(N=\mathcal C_G(M)\), then
\[
\kappa_M=\|G^{1/2}|_M\|
\le
\|G^{1/2}|_N\|
=
\sqrt{\lambda_N}.
\]
Thus \prettyref{eq:kato-finite-bound-refined} recovers the two forms of
the estimate in \prettyref{thm:2-2}. The only additional point in the
unbounded argument is the regularization
\(S(s)K\subseteq D(A^{1/2})\) for \(s>0\).
\end{rem}

\section{Nonautonomous NTK evolution and piecewise-frozen approximation}\label{sec:na-ntk}

We now return to the semigroup perturbation framework and allow the
Jacobian to vary with time. The fixed-kernel model from \prettyref{sec:2}
is the basic lazy-training approximation; in the infinite-width regime,
the NTK becomes effectively constant and the residual obeys a linear
differential equation in some function space
\cite{jacot2018neural,lee2019wide,chizat2019lazy}. At finite width,
however, the empirical NTK may evolve substantially during training; see
\cite{huang2020dynamics,seleznova2022analyzing,vyas2022limitations,atanasov2022silent,loo2022evolution}
for theoretical and empirical accounts of this kernel evolution. These
results motivate a piecewise-frozen model; instead of holding one
Jacobian $T$ fixed on all of $[0,\tau]$, we split the interval into
subintervals and freeze the linearization separately on each
subinterval. This approach is natural from the viewpoint of
nonautonomous evolution equations and product formulas
\cite{batkai2011operator,batkai2012norm}, and is consistent with the
product-integral treatment of time-varying kernels used by Boix-Adsera
and Littwin in their analysis of NTK validity \cite{boix2023tight}.
Throughout this section, fix $\tau>0$, and assume that
\[
T(\cdot):[0,\tau]\to\mathcal{B}(H,K)
\]
is norm-continuous and bounded. Set
\[
G(t):=T(t)T(t)^{*},\qquad 0\le t\le\tau.
\]
Then each $G(t)$ is bounded, self-adjoint, and positive on $K$. We consider
the nonautonomous residual equation
\begin{equation}
\label{eq:na-1}
r'(t)=-G(t)r(t),\qquad r(0)=r_{0}\in K.
\end{equation}
By the standard theory for bounded nonautonomous evolution equations, there is a
unique evolution family $U(t,s)$, $0\le s\le t\le\tau$, such that
\begin{equation}
\label{eq:na-2}
r(t)=U(t,s)r(s),\qquad \partial_{t}U(t,s)=-G(t)U(t,s),\qquad U(s,s)=I.
\end{equation}
The positivity of $G(t)$ implies that $U(t,s)$ is contractive.

As in the frozen setting, we fix a closed task space $M\subseteq K$.
The time-dependent evolution may carry vectors initially in $M$ through
additional output directions. Over the interval $[0,\tau]$, the task therefore
activates not only $M$ itself, but all directions visited by trajectories
$t\mapsto U(t,0)v$ with $v\in M$. This motivates the finite-horizon
nonautonomous orbit hull
\begin{equation}
\label{eq:na-3}
M_{\tau}^{\mathrm{na}}:=\overline{\operatorname{span}}\left\{ U(t,0)v:v\in M,\ 0\le t\le\tau\right\} .
\end{equation}
Thus, $M_{\tau}^{\mathrm{na}}$ is the portion of the output space activated by
the task over $[0,\tau]$; it is the smallest closed linear subspace containing
every state reached from $M$ during that interval. Unlike the cyclic space
$\mathcal C_G(M)$ in the autonomous setting, $M_{\tau}^{\mathrm{na}}$ can
genuinely depend on the horizon: as $t$ varies, the generators $G(t)$ may
couple the task to new directions, and there is no single fixed generator whose
cyclic subspace captures the whole evolution. In the piecewise-frozen analysis
below, this principle is implemented more locally through the propagated task
spaces $M_k$ and the corresponding local orbit spaces $N_k$, on which the
intervalwise tangent-capture hypotheses are imposed.

\begin{thm}
\label{thm:na-contract}
For every $0\le s\le t\le\tau$, one has
\[
\|U(t,s)\|\le1.
\]
\end{thm}

\begin{proof}
Fix $u\in K$, and set $w(t):=U(t,s)u$. Then $w'(t)=-G(t)w(t)$, so
\[
\frac{d}{dt}\|w(t)\|_{K}^{2}=2\Re\langle w'(t),w(t)\rangle_{K}
=-2\langle G(t)w(t),w(t)\rangle_{K}\le0.
\]
Hence $\|w(t)\|_{K}\le\|u\|_{K}$ for $t\ge s$, and taking the supremum over
unit vectors $u$ gives the conclusion.
\end{proof}

We next approximate the nonautonomous dynamics by freezing the Jacobian on each
subinterval of a partition.
Let
\[
\Pi:\qquad0=t_{0}<t_{1}<\cdots<t_{N}=\tau
\]
be a partition, set $\Delta t_{n}:=t_{n+1}-t_{n}$, and choose reference points
$\xi_{n}\in[t_{n},t_{n+1}]$. Define
\begin{equation}
\label{eq:na-4}
T_{n}:=T(\xi_{n}),\qquad G_{n}:=T_{n}T_{n}^{*},\qquad 0\le n\le N-1.
\end{equation}
For $m=1,\dots,N$, define the piecewise-frozen propagator by
\begin{equation}
\label{eq:na-5}
U_{\Pi}(t_{m},0):=e^{-\Delta t_{m-1}G_{m-1}}\cdots e^{-\Delta t_{0}G_{0}}.
\end{equation}
More generally, if $0\le k<m\le N$, set
\begin{equation}
\label{eq:na-6}
U_{\Pi}(t_{m},t_{k}):=e^{-\Delta t_{m-1}G_{m-1}}\cdots e^{-\Delta t_{k}G_{k}},
\qquad U_{\Pi}(t_{k},t_{k}):=I.
\end{equation}

\begin{thm}
\label{thm:na-freezing}
For every $m=1,\dots,N$,
\begin{equation}
\label{eq:na-7}
\left\|\left(U(t_{m},0)-U_{\Pi}(t_{m},0)\right)\big|_{M}\right\|
\le
\sum_{k=0}^{m-1}\int_{t_{k}}^{t_{k+1}}\|G(s)-G_{k}\|\,ds.
\end{equation}
Consequently,
\begin{equation}
\label{eq:na-8}
\sup_{0\le m\le N}\left\|\left(U(t_{m},0)-U_{\Pi}(t_{m},0)\right)\big|_{M}\right\|
\le
\sum_{k=0}^{N-1}\int_{t_{k}}^{t_{k+1}}\|G(s)-G_{k}\|\,ds.
\end{equation}
\end{thm}

\begin{proof}
Fix $k\in\{0,\dots,N-1\}$, and define
\[
V(s):=U(t_{k+1},s)e^{-(s-t_{k})G_{k}},\qquad s\in[t_{k},t_{k+1}].
\]
Since $\partial_{s}U(t_{k+1},s)=U(t_{k+1},s)G(s)$ and
$\partial_{s}e^{-(s-t_{k})G_{k}}=-G_{k}e^{-(s-t_{k})G_{k}}$, we get
\[
V'(s)=U(t_{k+1},s)\left(G(s)-G_{k}\right)e^{-(s-t_{k})G_{k}}.
\]
Integrating from $t_{k}$ to $t_{k+1}$ yields
\begin{equation}
\label{eq:na-9}
U(t_{k+1},t_{k})-e^{-\Delta t_{k}G_{k}}
=
\int_{t_{k}}^{t_{k+1}}U(t_{k+1},s)\left(G_{k}-G(s)\right)e^{-(s-t_{k})G_{k}}\,ds.
\end{equation}
By Theorem~\ref{thm:na-contract} and the positivity of $G_{k}$,
\begin{equation}
\label{eq:na-10}
\left\|U(t_{k+1},t_{k})-e^{-\Delta t_{k}G_{k}}\right\|
\le
\int_{t_{k}}^{t_{k+1}}\|G(s)-G_{k}\|\,ds.
\end{equation}
Now write
\[
U(t_{m},0)=U(t_{m},t_{m-1})\cdots U(t_{1},t_{0}),
\qquad
U_{\Pi}(t_{m},0)=e^{-\Delta t_{m-1}G_{m-1}}\cdots e^{-\Delta t_{0}G_{0}}.
\]
The telescoping product identity gives
\begin{align*}
U(t_{m},0)-U_{\Pi}(t_{m},0)
&=\sum_{k=0}^{m-1}
U(t_{m},t_{k+1})\left(U(t_{k+1},t_{k})-e^{-\Delta t_{k}G_{k}}\right)
U_{\Pi}(t_{k},0).
\end{align*}
All factors $U(t_{m},t_{k+1})$ and $U_{\Pi}(t_{k},0)$ are contractions, so
\begin{align*}
\left\|\left(U(t_{m},0)-U_{\Pi}(t_{m},0)\right)\big|_{M}\right\|
&\le\sum_{k=0}^{m-1}
\left\|U(t_{k+1},t_{k})-e^{-\Delta t_{k}G_{k}}\right\|.
\end{align*}
Combining this with \prettyref{eq:na-10} gives \prettyref{eq:na-7}. The bound
\prettyref{eq:na-8} follows by taking the supremum over $m$.
\end{proof}

\begin{cor}
\label{cor:na-holder}
Assume that there exist $\kappa>0$ and $\alpha\in(0,1]$ such that
\begin{equation}
\label{eq:na-11}
\|G(t)-G(s)\|\le \kappa|t-s|^{\alpha},\qquad 0\le s,t\le\tau.
\end{equation}
If the reference points are chosen as midpoints,
\[
\xi_{k}=\frac{t_{k}+t_{k+1}}{2},\qquad 0\le k\le N-1,
\]
then
\begin{equation}
\label{eq:na-12}
\sup_{0\le m\le N}\left\|\left(U(t_{m},0)-U_{\Pi}(t_{m},0)\right)\big|_{M}\right\|
\le
\frac{\kappa\tau}{2^{\alpha}(\alpha+1)}|\Pi|^{\alpha},
\end{equation}
where
\[
|\Pi|:=\max_{0\le k\le N-1}\Delta t_{k}.
\]
\end{cor}

\begin{proof}
By \prettyref{thm:na-freezing},
\[
\left\|\left(U(t_{m},0)-U_{\Pi}(t_{m},0)\right)\big|_{M}\right\|
\le
\sum_{k=0}^{m-1}\int_{t_{k}}^{t_{k+1}}\kappa|s-\xi_{k}|^{\alpha}\,ds.
\]
Since $\xi_{k}$ is the midpoint of $[t_{k},t_{k+1}]$,
\[
\int_{t_{k}}^{t_{k+1}}|s-\xi_{k}|^{\alpha}\,ds
=2\int_{0}^{\Delta t_{k}/2}x^{\alpha}\,dx
=\frac{\Delta t_{k}^{\alpha+1}}{2^{\alpha}(\alpha+1)}
\le
\frac{|\Pi|^{\alpha}\Delta t_{k}}{2^{\alpha}(\alpha+1)}.
\]
Summing over $k$ gives \prettyref{eq:na-12}.
\end{proof}

We now add, on each subinterval, a restriction of parameter directions analogous
to the fixed-kernel perturbation studied earlier.
Let $P_{0},\dots,P_{N-1}$ be orthogonal projections on $H$, and define
\begin{equation}
\label{eq:na-13}
G_{n,P}:=T_{n}P_{n}T_{n}^{*},\qquad
q_{n}:=\|G_{n}-G_{n,P}\|,\qquad 0\le n\le N-1.
\end{equation}
Set
\begin{equation}
\label{eq:na-14}
U_{\Pi,P}(t_{m},0):=e^{-\Delta t_{m-1}G_{m-1,P}}\cdots e^{-\Delta t_{0}G_{0,P}},
\qquad m=1,\dots,N.
\end{equation}
Define recursively the propagated task spaces
\begin{equation}
\label{eq:na-15}
M_{0}:=M,
\qquad
M_{k}:=\overline{U_{\Pi}(t_{k},0)M},\qquad 0\le k\le N,
\end{equation}
and the local orbit-generated spaces
\begin{equation}
\label{eq:na-16}
N_{k}:=\overline{\operatorname{span}}\left\{ e^{-sG_{k}}v:v\in M_{k},\ 0\le s\le\Delta t_{k}\right\},
\qquad 0\le k\le N-1.
\end{equation}
Finally, set
\begin{equation}
\label{eq:na-17}
\lambda_{k}:=\|G_{k}|_{N_{k}}\|,\qquad 0\le k\le N-1.
\end{equation}

\begin{thm}
\label{thm:na-piecewise-prune}
Assume that for each $k=0,\dots,N-1$ there exists $\epsilon_{k}\in[0,1)$ such that
\begin{equation}
\label{eq:na-18}
\|(I-P_{k})T_{k}^{*}u\|_{H}\le\epsilon_{k}\|T_{k}^{*}u\|_{H},\qquad \forall u\in N_{k}.
\end{equation}
Then for every $m=1,\dots,N$,
\begin{equation}
\label{eq:na-19}
\left\|\left(U_{\Pi,P}(t_{m},0)-U_{\Pi}(t_{m},0)\right)\big|_{M}\right\|
\le
\sum_{k=0}^{m-1}\epsilon_{k}
\min\left\{
\Delta t_{k}\sqrt{q_{k}\lambda_{k}},\sqrt{\Delta t_{k}q_{k}/2}
\right\} .
\end{equation}
Consequently,
\begin{equation}
\label{eq:na-20}
\sup_{0\le m\le N}\left\|\left(U_{\Pi,P}(t_{m},0)-U_{\Pi}(t_{m},0)\right)\big|_{M}\right\|
\le
\sum_{k=0}^{N-1}\epsilon_{k}
\min\left\{
\Delta t_{k}\sqrt{q_{k}\lambda_{k}},\sqrt{\Delta t_{k}q_{k}/2}
\right\} .
\end{equation}
\end{thm}

\begin{proof}
Fix $k$. Set
\[
E_{k}:=e^{-\Delta t_{k}G_{k}},\qquad E_{k,P}:=e^{-\Delta t_{k}G_{k,P}}.
\]
Since $G_{k}\ge0$ and $G_{k,P}\ge0$, both $E_{k}$ and $E_{k,P}$ are contractions.
By Duhamel's formula, for $0\le t\le\Delta t_{k}$,
\[
e^{-tG_{k,P}}-e^{-tG_{k}}
=\int_{0}^{t}e^{-(t-s)G_{k,P}}\left(G_{k}-G_{k,P}\right)e^{-sG_{k}}\,ds.
\]
Hence for every $v\in M_{k}$,
\begin{equation}
\label{eq:na-21}
\left\|\left(e^{-tG_{k,P}}-e^{-tG_{k}}\right)v\right\|_{K}
\le
\int_{0}^{t}\left\|\left(G_{k}-G_{k,P}\right)e^{-sG_{k}}v\right\|_{K}\,ds.
\end{equation}
Now
\[
\left(G_{k}-G_{k,P}\right)u=T_{k}(I-P_{k})T_{k}^{*}u,
\]
and, as in the proof of \prettyref{thm:2-2},
\[
G_{k}-G_{k,P}=\big[T_{k}(I-P_{k})\big]
\big[T_{k}(I-P_{k})\big]^{*}.
\]
Thus, if $u\in N_{k}$, then by \prettyref{eq:na-18},
\begin{align}
\left\|\left(G_{k}-G_{k,P}\right)u\right\|_{K}
&\le\sqrt{q_{k}}\,\|(I-P_{k})T_{k}^{*}u\|_{H}\nonumber\\
&\le\epsilon_{k}\sqrt{q_{k}}\,\|T_{k}^{*}u\|_{H}
=\epsilon_{k}\sqrt{q_{k}}\,\|G_{k}^{1/2}u\|_{K}.\label{eq:na-22}
\end{align}
Since $v\in M_{k}$ and $0\le s\le\Delta t_{k}$ imply $e^{-sG_{k}}v\in N_{k}$, we may use
\prettyref{eq:na-22} in \prettyref{eq:na-21}. This gives
\begin{equation}
\label{eq:na-23}
\left\|\left(e^{-tG_{k,P}}-e^{-tG_{k}}\right)v\right\|_{K}
\le
\epsilon_{k}\sqrt{q_{k}}\int_{0}^{t}\|G_{k}^{1/2}e^{-sG_{k}}v\|_{K}\,ds.
\end{equation}
We estimate the integral in two ways.
First, since $e^{-sG_{k}}v\in N_{k}$,
\[
\|G_{k}^{1/2}e^{-sG_{k}}v\|_{K}
\le\|G_{k}^{1/2}|_{N_{k}}\|\,\|e^{-sG_{k}}v\|_{K}
\le\sqrt{\lambda_{k}}\,\|v\|_{K}.
\]
Substituting into \prettyref{eq:na-23} yields
\begin{equation}
\label{eq:na-24}
\left\|\left(e^{-tG_{k,P}}-e^{-tG_{k}}\right)v\right\|_{K}
\le
\epsilon_{k}t\sqrt{q_{k}\lambda_{k}}\,\|v\|_{K}.
\end{equation}
Second, by Cauchy--Schwarz,
\[
\int_{0}^{t}\|G_{k}^{1/2}e^{-sG_{k}}v\|_{K}\,ds
\le\sqrt{t}\left(\int_{0}^{t}\|G_{k}^{1/2}e^{-sG_{k}}v\|_{K}^{2}\,ds\right)^{1/2}.
\]
But
\[
\frac{d}{ds}\|e^{-sG_{k}}v\|_{K}^{2}
=-2\|G_{k}^{1/2}e^{-sG_{k}}v\|_{K}^{2},
\]
so
\[
\int_{0}^{t}\|G_{k}^{1/2}e^{-sG_{k}}v\|_{K}^{2}\,ds
=\frac12\left(\|v\|_{K}^{2}-\|e^{-tG_{k}}v\|_{K}^{2}\right)
\le\frac12\|v\|_{K}^{2}.
\]
Hence
\begin{equation}
\label{eq:na-25}
\left\|\left(e^{-tG_{k,P}}-e^{-tG_{k}}\right)v\right\|_{K}
\le
\epsilon_{k}\sqrt{tq_{k}/2}\,\|v\|_{K}.
\end{equation}
Taking $t=\Delta t_{k}$ in \prettyref{eq:na-24} and \prettyref{eq:na-25}, we obtain
\begin{equation}
\label{eq:na-26}
\|(E_{k,P}-E_{k})|_{M_{k}}\|
\le
\epsilon_{k}\min\left\{
\Delta t_{k}\sqrt{q_{k}\lambda_{k}},\sqrt{\Delta t_{k}q_{k}/2}
\right\} .
\end{equation}
Now use the telescoping product identity
\[
U_{\Pi,P}(t_{m},0)-U_{\Pi}(t_{m},0)
=\sum_{k=0}^{m-1}
U_{\Pi,P}(t_{m},t_{k+1})(E_{k,P}-E_{k})U_{\Pi}(t_{k},0).
\]
All factors $U_{\Pi,P}(t_{m},t_{k+1})$ and $U_{\Pi}(t_{k},0)$ are contractions. Since
$U_{\Pi}(t_{k},0)M\subseteq M_{k}$, \prettyref{eq:na-26} gives
\[
\left\|\left(U_{\Pi,P}(t_{m},0)-U_{\Pi}(t_{m},0)\right)\big|_{M}\right\|
\le
\sum_{k=0}^{m-1}\|(E_{k,P}-E_{k})|_{M_{k}}\|,
\]
which is exactly \prettyref{eq:na-19}. Taking the supremum over $m$ gives
\prettyref{eq:na-20}.
\end{proof}

\begin{cor}
\label{cor:na-combined}
Under the hypotheses of \prettyref{thm:na-freezing} and
\prettyref{thm:na-piecewise-prune}, one has for every $m=1,\dots,N$,
\begin{align}
\left\|\left(U(t_{m},0)-U_{\Pi,P}(t_{m},0)\right)\big|_{M}\right\|
&\le
\sum_{k=0}^{m-1}\int_{t_{k}}^{t_{k+1}}\|G(s)-G_{k}\|\,ds\nonumber\\
&\quad+
\sum_{k=0}^{m-1}\epsilon_{k}
\min\left\{
\Delta t_{k}\sqrt{q_{k}\lambda_{k}},\sqrt{\Delta t_{k}q_{k}/2}
\right\} .
\label{eq:na-27}
\end{align}
In particular,
\begin{align}
\sup_{0\le m\le N}\left\|\left(U(t_{m},0)-U_{\Pi,P}(t_{m},0)\right)\big|_{M}\right\|
&\le
\sum_{k=0}^{N-1}\int_{t_{k}}^{t_{k+1}}\|G(s)-G_{k}\|\,ds\nonumber\\
&\quad+
\sum_{k=0}^{N-1}\epsilon_{k}
\min\left\{
\Delta t_{k}\sqrt{q_{k}\lambda_{k}},\sqrt{\Delta t_{k}q_{k}/2}
\right\} .
\label{eq:na-28}
\end{align}
If, in addition, \prettyref{eq:na-11} holds and the $\xi_{k}$ are midpoints, then
\begin{align}
\sup_{0\le m\le N}\left\|\left(U(t_{m},0)-U_{\Pi,P}(t_{m},0)\right)\big|_{M}\right\|
&\le
\frac{\kappa\tau}{2^{\alpha}(\alpha+1)}|\Pi|^{\alpha}\nonumber\\
&\quad+
\sum_{k=0}^{N-1}\epsilon_{k}
\min\left\{
\Delta t_{k}\sqrt{q_{k}\lambda_{k}},\sqrt{\Delta t_{k}q_{k}/2}
\right\} .
\label{eq:na-29}
\end{align}
\end{cor}


\begin{rem}
\label{rem:na-local-refine}
Assume, in addition, that $K$ is finite-dimensional and that
$G_{k,P}N_k\subseteq N_k$ for every $k$. Then the decomposition developed
earlier in the paper into diagonal attenuation and off-diagonal mixing may also
be applied interval-by-interval. Namely, one may decompose each local operator
$G_{k,P}|_{N_{k}}$ relative to
the spectral decomposition of $G_{k}|_{N_{k}}$, estimate the local diagonal and
mixing terms exactly as in the frozen case, and then sum the resulting local
bounds over $k$ by the same telescoping argument used above. Thus the earlier
refinement survives in the piecewise-frozen setting with no essential change in
the proof.
\end{rem}

\section{A discrete Euler specialization}\label{subsec:na-discrete}

In Section~\ref{sec:na-ntk}, the nonautonomous problem through products of local semigroup factors $e^{-\Delta t_k G_k}$ was considered. We now pass to the corresponding discrete-time model by replacing each such factor with one explicit Euler step $I-\eta_k G_k$, and we derive the parallel perturbation estimate for the pruned products.
When the step size is not small, gradient descent is a discrete
iteration rather than a flow, and the discrete path need not be well
captured by the continuous-time model
\cite{bell2023exact}. 

In this section, we present the output-side Euler products and the corresponding
comparison estimate. Let
$\eta_{0},\dots,\eta_{N-1}>0$ and define
\begin{equation}
\label{eq:na-30}
R_{n+1}=(I-\eta_{n}G_{n})R_{n},
\qquad
R_{n+1}^{P}=(I-\eta_{n}G_{n,P})R_{n}^{P},
\qquad 0\le n\le N-1.
\end{equation}
Equivalently,
\begin{equation}
\label{eq:na-31}
V_{\Pi}(m,0):=(I-\eta_{m-1}G_{m-1})\cdots(I-\eta_{0}G_{0}),
\end{equation}
\begin{equation}
\label{eq:na-32}
V_{\Pi,P}(m,0):=(I-\eta_{m-1}G_{m-1,P})\cdots(I-\eta_{0}G_{0,P}).
\end{equation}
Assume that
\begin{equation}
\label{eq:na-33}
\eta_{n}\|G_{n}\|\le2,
\qquad
\eta_{n}\|G_{n,P}\|\le2,
\qquad 0\le n\le N-1,
\end{equation}
so that each factor in \prettyref{eq:na-31}--\prettyref{eq:na-32} is a
contraction.
Define the propagated discrete task spaces
\begin{equation}
\label{eq:na-34}
M_{0}^{E}:=M,
\qquad
M_{k}^{E}:=\overline{V_{\Pi}(k,0)M},\qquad 0\le k\le N,
\end{equation}
and set
\begin{equation}
\label{eq:na-35}
\lambda_{k}^{E}:=\|G_{k}|_{M_{k}^{E}}\|,\qquad 0\le k\le N-1.
\end{equation}

\begin{thm}
\label{thm:na-discrete}
Assume that for each $k=0,\dots,N-1$ there exists $\epsilon_{k}\in[0,1)$, such that
\begin{equation}
\label{eq:na-36}
\|(I-P_{k})T_{k}^{*}u\|_{H}\le\epsilon_{k}\|T_{k}^{*}u\|_{H},
\qquad \forall u\in M_{k}^{E}.
\end{equation}
Then for every $m=1,\dots,N$,
\begin{equation}
\label{eq:na-37}
\left\|\left(V_{\Pi,P}(m,0)-V_{\Pi}(m,0)\right)\big|_{M}\right\|
\le
\sum_{k=0}^{m-1}\eta_{k}\epsilon_{k}\sqrt{q_{k}\lambda_{k}^{E}}.
\end{equation}
Consequently,
\begin{equation}
\label{eq:na-38}
\sup_{0\le m\le N}\left\|\left(V_{\Pi,P}(m,0)-V_{\Pi}(m,0)\right)\big|_{M}\right\|
\le
\sum_{k=0}^{N-1}\eta_{k}\epsilon_{k}\sqrt{q_{k}\lambda_{k}^{E}}.
\end{equation}
\end{thm}

\begin{proof}
For $u\in M_{k}^{E}$,
\[
(G_{k}-G_{k,P})u=T_{k}(I-P_{k})T_{k}^{*}u.
\]
Moreover, since $G_{k}\ge0$, the Cauchy--Schwarz inequality and
\prettyref{eq:na-35} give
\[
\|G_{k}^{1/2}u\|_{K}^{2}
=\langle G_{k}u,u\rangle_{K}
\le\|G_{k}u\|_{K}\|u\|_{K}
\le\lambda_{k}^{E}\|u\|_{K}^{2}.
\]
Using the same factorization as in \prettyref{eq:na-22} and then
\prettyref{eq:na-36},
\begin{align*}
\|(G_{k}-G_{k,P})u\|_{K}
&\le\sqrt{q_{k}}\,\|(I-P_{k})T_{k}^{*}u\|_{H}\\
&\le\epsilon_{k}\sqrt{q_{k}}\,\|G_{k}^{1/2}u\|_{K}\\
&\le\epsilon_{k}\sqrt{q_{k}\lambda_{k}^{E}}\,\|u\|_{K}.
\end{align*}
Hence,
\begin{equation}
\label{eq:na-39}
\|(G_{k}-G_{k,P})|_{M_{k}^{E}}\|
\le
\epsilon_{k}\sqrt{q_{k}\lambda_{k}^{E}}.
\end{equation}
Now apply the telescoping identity
\[
V_{\Pi,P}(m,0)-V_{\Pi}(m,0)
=\sum_{k=0}^{m-1}
V_{\Pi,P}(m,k+1)\eta_{k}(G_{k}-G_{k,P})V_{\Pi}(k,0).
\]
By \prettyref{eq:na-33}, all factors $V_{\Pi,P}(m,k+1)$ and $V_{\Pi}(k,0)$ are
contractions. Since $V_{\Pi}(k,0)M\subseteq M_{k}^{E}$, \prettyref{eq:na-39}
yields
\[
\left\|\left(V_{\Pi,P}(m,0)-V_{\Pi}(m,0)\right)\big|_{M}\right\|
\le
\sum_{k=0}^{m-1}\eta_{k}\epsilon_{k}\sqrt{q_{k}\lambda_{k}^{E}},
\]
which is \prettyref{eq:na-37}. Taking the supremum over $m$ gives
\prettyref{eq:na-38}.
\end{proof}

The Euler products contain, as a special case, output-side projection
algorithms of Kaczmarz type.

\begin{cor}
\label{cor:na-kaczmarz}
Suppose that, for each $k$, the local kernel operator has the form
\begin{equation}
\label{eq:na-40}
G_{k}=\lambda_{k}Q_{k},
\qquad \lambda_{k}\ge0,
\end{equation}
where $Q_{k}$ is an orthogonal projection on $K$. Then the Euler step in
\prettyref{eq:na-30} becomes
\begin{equation}
\label{eq:na-41}
R_{k+1}=(I-\eta_{k}\lambda_{k}Q_{k})R_{k}.
\end{equation}
If $\lambda_{k}>0$ and $\eta_{k}=\lambda_{k}^{-1}$, then
\begin{equation}
\label{eq:na-42}
R_{k+1}=(I-Q_{k})R_{k},
\end{equation}
which is the classical projection step on $K$. More generally,
$0<\eta_{k}\lambda_{k}\le2$ gives the relaxed Kaczmarz step. In particular, when
the local NTK operators are rank-one or low-rank projections, the
piecewise-frozen dynamics become a kernel Kaczmarz or block-Kaczmarz
product on the output side.
\end{cor}

\begin{proof}
This is immediate from \prettyref{eq:na-30} after substituting
\prettyref{eq:na-40}.
\end{proof}

\begin{rem}
\label{rem:na-kaczmarz-lit}
The output-side product in \prettyref{eq:na-41} is close in spirit to the
projection-product viewpoint in the Kaczmarz literature developed by
Jorgensen, Song, and Tian, and, more recently, to infinite-dimensional
operator and block Kaczmarz methods with explicit $\lambda$-dependent
bounds \cite{jorgensen2021positive,jorgensen2026defect,jeong2025infinite}.
It should be regarded as a special projection example on the output
side. For related recent advances on infinite-dimensional operator
Kaczmarz methods, see \cite{MR4835160,MR4472249,MR3996038}.
\end{rem}

\section*{Acknowledgements}
Jorgensen wishes to thank his present and recent co-authors, Daniel
Alpay, Sergii Bezuglyi, David
Stewart, and Eric Weber, for many illuminating discussions covering
diverse topics related to the theoretic foundations of the present
paper.
Jeong thanks Eric Weber for the insightful discussion on Kaczmarz methods.

\appendix

\section{A worked example}\label{sec:6}

In the Appendix, we give a concrete finite-dimensional example showing how the
constants in Theorems~\ref{thm:2-2} and \ref{thm:3-1} may be
computed explicitly. The point of the example is twofold. First, the
Jacobian $T$ is not diagonal, so the model does not collapse to a
trivial coordinate aligned case. Second, the cyclic task space generated
by $M$ occupies only part of the spectrum of $G$, and the example illustrates both the
discarded-kernel estimate from \prettyref{sec:3} and the complementary
attenuation/mixing refinement from \prettyref{sec:4}.
Let $H=\mathbb{R}^{4}$, $K=\mathbb{R}^{3}$, and define
\[
T=\begin{pmatrix}6 & 0 & 0 & 0\\
0 & \sqrt{3} & 0 & 1\\
0 & -\frac{1}{2\sqrt{3}} & \sqrt{\frac{2}{3}} & \frac{1}{2}
\end{pmatrix}.
\]
Let $P=diag\left(1,1,1,0\right)$, so that one of the four parameter
directions is removed. Then
\[
G=TT^{*}=\begin{pmatrix}36 & 0 & 0\\
0 & 4 & 0\\
0 & 0 & 1
\end{pmatrix},\quad G_{P}=TPT^{*}=\begin{pmatrix}36 & 0 & 0\\
0 & 3 & -\frac{1}{2}\\
0 & -\frac{1}{2} & \frac{3}{4}
\end{pmatrix}.
\]
The discarded part of the kernel is
\[
Q_{P}=G-G_{P}=\begin{pmatrix}0 & 0 & 0\\
0 & 1 & \frac{1}{2}\\
0 & \frac{1}{2} & \frac{1}{4}
\end{pmatrix},
\qquad
\|Q_{P}\|=\frac{5}{4}.
\]
Take the task space
\[
M=\mathrm{span}\left\{ e_{2},e_{3}\right\} \subseteq K.
\]
Since $G$ is diagonal in the standard basis, $M$ is $G$-invariant.
Hence,
\[
\mathcal C_G(M)=M.
\]
Therefore, the minimal $G$-invariant subspace generated by $M$ is
\[
N=\mathcal C_G(M)=M,
\]
and
\[
\lambda_{N}=\Vert G|_{N}\Vert=4,\qquad\Vert G\Vert=36.
\]
In particular, the task space sees only the eigenvalues $4$ and $1$,
while the global operator norm of $G$ is governed by the inactive
mode $36$. The finite-time perturbation estimate depends instead on
the smaller discarded-kernel norm $\|Q_{P}\|=5/4$.
We first compute the constant $\epsilon$ from \prettyref{eq:2-2}.
For $u=\left(0,x,y\right)\in N$, one has
\[
T^{*}u=\begin{pmatrix}0\\
\sqrt{3}\,x-\frac{1}{2\sqrt{3}}y\\
\sqrt{\frac{2}{3}}\,y\\
x+\frac{1}{2}y
\end{pmatrix}.
\]
A direct calculation gives $\Vert T^{*}u\Vert^{2}=4x^{2}+y^{2}$.
Also,
\[
\left(I-P\right)T^{*}u=\begin{pmatrix}0\\
0\\
0\\
x+\frac{1}{2}y
\end{pmatrix},
\]
so
\[
\Vert\left(I-P\right)T^{*}u\Vert^{2}=\left(x+\frac{1}{2}y\right)^{2}\le\frac{1}{2}\left(4x^{2}+y^{2}\right)\qquad\forall x,y\in\mathbb{R}.
\]
Thus, the best constant in \prettyref{eq:2-2} is
\[
\epsilon=\frac{1}{\sqrt{2}}.
\]
Applying \prettyref{thm:2-2} with $N=\mathcal C_G(M)=M$, $\lambda_{N}=4$,
and $\Vert Q_{P}\Vert=5/4$, we obtain
\[
\left\Vert \left(S_{P}\left(t\right)-S\left(t\right)\right)\big|_{M}\right\Vert
\le\frac{1}{\sqrt{2}}\min\left\{ \sqrt{5}\,t,\sqrt{\frac{5t}{8}}\right\} ,\qquad t\ge0.
\]
We now compute the refined constants from \prettyref{sec:4}. Relative
to $N=\mathrm{span}\left\{ e_{2},e_{3}\right\} $, one has
\[
G_{P}|_{N}=\begin{pmatrix}3 & -\frac{1}{2}\\
-\frac{1}{2} & \frac{3}{4}
\end{pmatrix}.
\]
Since $G|_{N}=\mathrm{diag}\left(4,1\right)$, the decomposition from
\prettyref{eq:3-3} is
\[
D=\begin{pmatrix}3 & 0\\
0 & \frac{3}{4}
\end{pmatrix},\qquad R=\begin{pmatrix}0 & -\frac{1}{2}\\
-\frac{1}{2} & 0
\end{pmatrix}, 
\]
and hence,
\[
\rho=\Vert R\Vert=\frac{1}{2}.
\]
Moreover,
\[
G|_{N}-D=\begin{pmatrix}1 & 0\\
0 & \frac{1}{4}
\end{pmatrix}\le\frac{1}{4}\begin{pmatrix}4 & 0\\
0 & 1
\end{pmatrix}=\frac{1}{4}G|_{N},
\]
so \prettyref{eq:3-5} holds with
\[
\delta=\frac{1}{4}.
\]
Therefore, \prettyref{thm:3-1} yields
\[
\left\Vert \left(S_{P}\left(t\right)-S\left(t\right)\right)\big|_{M}\right\Vert \le1-e^{-t}+\frac{t}{2},\quad t\ge0.
\]

In this example, the diagonal contribution can, in fact, be computed
exactly, since both $D$ and $G|_{N}$ are diagonal in the basis $\left\{ e_{2},e_{3}\right\} $.
Indeed,
\[
\Vert e^{-tD}-e^{-tG|_{N}}\Vert=\max\left\{ e^{-3t}-e^{-4t},e^{-3t/4}-e^{-t}\right\}.
\]
Hence, the argument in the proof of \prettyref{thm:3-1} gives the
following sharper version of the Section~\ref{sec:4} estimate:
\[
\left\Vert \left(S_{P}\left(t\right)-S\left(t\right)\right)|_{M}\right\Vert
\le
\max\left\{ e^{-3t}-e^{-4t},e^{-3t/4}-e^{-t}\right\} +\frac{t}{2},
\quad t\ge0.
\]
The first term in the maximum comes from the $e_{2}$-mode and the second
from the $e_{3}$-mode; together they represent the diagonal attenuation,
while the term $t/2$ comes from the off-diagonal mixing. Thus, the
decomposition in \prettyref{sec:4} is concrete in this example.

Finally, we compare the actual operator norm $\Vert\left(S_{P}\left(t\right)-S\left(t\right)\right)|_{M}\Vert$
with the bounds from Theorems \ref{thm:2-2} and \ref{thm:3-1}. For
$t\in\left\{ 1/4,1/2,1\right\} $, one obtains the following values:
\begin{center}
\begin{tabular}{c|c|c|c}
$t$  & actual value  & \prettyref{thm:2-2}  & \prettyref{thm:3-1}\tabularnewline
\hline
$1/4$  & $0.1661$  & $0.2795$  & $0.3462$\tabularnewline
$1/2$  & $0.2014$  & $0.3953$  & $0.6435$\tabularnewline
$1$  & $0.1993$  & $0.5590$  & $1.1321$ \tabularnewline
\end{tabular}
\par\end{center}

\begin{rem}
In the example above, the constants $\epsilon$, $\delta$, and $\rho$
are optimal for the hypotheses of Theorems \ref{thm:2-2} and \ref{thm:3-1}.
More precisely, $\epsilon$ is the best constant such that
\[
\Vert\left(I-P\right)T^{*}u\Vert\le\epsilon\Vert T^{*}u\Vert, \qquad\forall u\in N,
\]
$\delta$ is the smallest constant such that
\[
0\le G|_{N}-D\le\delta G|_{N},
\]
and $\rho=\Vert R\Vert$ is computed exactly. The discarded-kernel norm
$\|Q_{P}\|=5/4$ is also exact. The estimates in Sections \ref{sec:3} and
\ref{sec:4} are complementary. The estimate in \prettyref{sec:4}
does not use the relative task-capture condition \prettyref{eq:2-2};
under the finite-dimensionality and $G_{P}$-invariance assumptions on
$N$, it instead separates diagonal attenuation, measured by $\delta$,
from off-diagonal mixing, measured by $\rho$. Its $O(t)$ behavior is
naturally suited to short-time comparison and can remain informative on
longer intervals when $\delta\lambda_{N}+\rho$ is small. By contrast,
\prettyref{thm:2-2} exploits a small task-capture constant $\epsilon$,
and its $O(\epsilon\sqrt{t})$ alternative can remain informative on
comparatively longer time intervals when $\epsilon$ is small.
\end{rem}

\bibliographystyle{amsplain}
\bibliography{ref_mod}

\end{document}